\documentclass[10pt]{amsart}
\usepackage[T1]{fontenc}
\usepackage{lmodern,amssymb,mathtools,mathrsfs,bm}
\usepackage{microtype,booktabs,graphicx,array,placeins}
\usepackage[hidelinks]{hyperref}
\hypersetup{pdftitle={A merged-cell discontinuous Galerkin method for scalar conservation laws on curved domains},
pdfsubject={Optimal error estimates and explicit RK3 time integration},
pdfkeywords={discontinuous Galerkin, conservation law, cut cell, cell merging, Runge--Kutta}}
\numberwithin{equation}{section}
\newtheorem{theorem}{Theorem}[section]
\newtheorem{lemma}[theorem]{Lemma}

\theoremstyle{definition}
\newtheorem{Def}{Definiton}[section]

\theoremstyle{remark}
\newtheorem{remark}[theorem]{Remark}
\newcommand{\R}{\mathbb R}

\newcommand{\F}{F}
\newcommand{\norm}[1]{\left\|#1\right\|}

\newcommand{\ip}[2]{\left(#1,#2\right)}
\newcommand{\epsh}{\varepsilon_h}
\newcommand{\Ph}{P_h}
\newcommand{\Lh}{L_h}
\newcommand{\Dh}{\mathcal{D}_h}
\newcommand{\Sh}{S_h}
\newcommand{\Fh}{\mathcal F_h}
\newcommand{\Vh}{V_h}
\newcommand{\Th}{\mathcal T_h}

\newcommand{\diam}{\operatorname{diam}}
\newcommand{\conv}{\operatorname{conv}}
\newcommand{\interior}{\operatorname{int}}
\DeclareMathOperator{\dist}{dist}
\newcommand{\cM}{\mathcal{M}_h}
\newcommand{\FF}{\mathbf F}
\newcommand{\jp}[1]{[\![#1]\!]}

\title[Merged-cell DG for conservation laws]{A merged-cell discontinuous Galerkin method for scalar conservation laws on curved domains}
\author{Yong Liu}
\address{ICMSEC, State Key Laboratory of Mathematical Sciences (SKLMS), Academy of Mathematics and Systems Science, and School of Mathematical Science, University of Chinese Academy of Sciences, Chinese Academy of Sciences, Beijing 100190, P.~R.~China.}
\email{yongliu@lsec.cc.ac.cn}
\thanks{Corresponding author. Research is partially supported by the Strategic Priority Research Program of the Chinese Academy of Sciences under the Grant No. XDB0640000, and the NSFC grant 12571395, 12288201, and the Youth Innovation Promotion Association (CAS)}
\date{}
\begin{document}
\begin{abstract}
We analyse a discontinuous Galerkin method for two-dimensional scalar
conservation laws on curved domains. Small intersections of a Cartesian
grid with the domain are combined by rectangular cell merging. The local
space is the restriction of $\mathbb Q_k$ to each physical merged cell.
An entropy-conservative flux is supplemented by face viscosity and, for
odd degrees, a skew derivative-jump correction on unmerged interior
squares. We prove exact semidiscrete $L^2$ dissipation and an optimal
error estimate for every fixed $k\ge1$. For smooth solutions with
homogeneous trace on the entire boundary, the semi-discrete method satisfies an $O(h^{k+1})$ error bound.
The constants are independent of the original cut-cell area fractions.
The proof combines a physical-cell $L^2$ projection and cancellation of
regular-face projection errors. A separate semidiscrete result covers linear transport with upwind inflow data. Computations for Burgers' equation on a disk and linear advection on a smooth star exhibit orders close to $k+1$ for
$k=1,2,3$.
\end{abstract}
\maketitle
\section{Introduction}

Discontinuous Galerkin (DG) methods combine local polynomial
approximations with numerical fluxes across cell interfaces. Their
local structure makes them particularly well suited to conservation
laws and explicit time integration; see the review by Cockburn and
Shu~\cite{CockburnShu2001}. For problems posed on curved domains,
unfitted mesh methods offer a geometrically flexible alternative to
traditional body-fitted discretizations. By allowing physical
boundaries and internal interfaces to cut through mesh elements
arbitrarily, these methods avoid the often costly construction of
boundary- and interface-fitted meshes~\cite{B1970Comp,CZ1998NM}.
Over the past several decades, unfitted methods have developed into
a versatile discretization framework for a wide range of applications,
including interface problems~\cite{BE1987IMANA,LL1994SINUM,BBH2009CMAME,BH2012ANM},
fluid dynamics~\cite{P1997JCP,SW2014CMAME},
surface partial differential equations~\cite{DER2014SINUM},
coupled bulk--surface systems~\cite{GO2015M2AN},
and problems on evolving domains~\cite{MZZ2022SINUM}.

Despite these advantages, unfitted finite element methods face a
fundamental challenge: the small-cut-cell problem. Arbitrarily small
cut cells can lead to ill-conditioned mass matrices, unfavorable
constants in inverse estimates, and severe time-step restrictions
for explicit time integration. Two main strategies have been
developed to address this difficulty. The first introduces suitable
stabilization techniques~\cite{B2010MASP,MLLR2014JSC,MLLR2014NM,BH2012ANM,BH2014M2AN,XXW2020CMAME}.
A prominent example is the ghost-penalty stabilization employed in
CutFEM~\cite{BCHLM2015IJNME}. For high-order finite elements, its
commonly used derivative-jump formulation requires the evaluation
of higher-order derivatives on mesh facets. The second strategy
merges small cut cells with neighboring larger
elements~\cite{BMV2018SISC,HWX2017CMAME,JL2013NM,BCDE2021SISC},
thereby forming macro-elements with sufficient support within the
physical domain. Chen and Liu~\cite{ChenLiu2023,CL2024ANM} developed a robust,
automated merging algorithm that efficiently constructs the
two-dimensional macro-elements introduced in~\cite{CLX2021NM}.
Building on this approach, the present paper addresses the
small-cut-cell problem through cell merging while retaining
a high-order error estimate for scalar conservation laws.

Even on fitted meshes, optimal error estimates for DG methods depend
on both the numerical flux and the mesh structure. Classical energy
estimates for linear transport exhibit a half-order loss relative to
the polynomial approximation order~\cite{JohnsonPitkaranta1986},
and quasiuniformity alone does not eliminate this
loss~\cite{Peterson1991}. For linear equations, projections tailored
to upwind-biased fluxes yield optimal error estimates in tensor-product
polynomial spaces on Cartesian meshes~\cite{MengShuWu2016}. Optimal
estimates in total-degree polynomial spaces on uniform
two-dimensional Cartesian meshes were studied by Liu, Shu, and
Zhang~\cite{LiuShuZhang2020} and, more recently, by Cao, Shu, and
Zhang~\cite{CaoShuZhang2026}. These analyses exploit the underlying
mesh structure and, in nonlinear settings, impose assumptions on
the wind direction. Central fluxes, by contrast, exhibit different
cancellation properties, and their accuracy can depend on both
polynomial-degree parity and mesh
uniformity~\cite{LiuShuZhang2021Central}. Extending projection
arguments from uniform square elements to general curved
macro-elements therefore requires additional estimates.

Several DG approaches have been developed to handle embedded and
curved boundaries in hyperbolic problems. Krivodonova and
Berger~\cite{KrivodonovaBerger2006} investigated high-order treatments
of curved solid walls, while Qin and
Krivodonova~\cite{QinKrivodonova2013} combined cell merging with an
explicit DG method for the Euler equations on embedded Cartesian
grids. Ghost-penalty stabilization provides cut-independent estimates
for stationary advection--reaction
problems~\cite{GurkanStickoMassing2020}, and Fu and
Kreiss~\cite{FuKreiss2021} developed high-order time-dependent cut-DG
methods in one dimension. Engwer et al.~\cite{EngwerEtAl2020}
introduced a conservative stabilization for linear transport that
permits time steps governed by the background grid rather than the
small cut cells. Giuliani~\cite{Giuliani2022} developed high-order
state redistribution on curvilinear embedded-boundary grids.
Taylor, Wilcox, and Chan~\cite{TaylorWilcoxChan2025} subsequently
proved that state redistribution can preserve the $L^2$ stability
of a cut-DG wave discretization.

Recent developments also address nonlinear stability. Fu, Kreiss,
and Zahedi~\cite{FuKreissZahedi2024} combined macroelement
reconstruction, stabilization, and limiting to preserve solution
bounds in one dimension. A multidimensional extension for scalar
conservation laws is presented in the
preprint~\cite{FuKreissXinZahedi2026}. Taylor and
Chan~\cite{TaylorChan2026} constructed entropy-stable high-order DG
operators on cut meshes using summation-by-parts formulations and
nonnegative quadrature. Their work distinguishes the entropy stability
of the spatial operator from stability of the small-cut-cell
treatment.

We adopt the rectangular cell-merging construction of Chen and
Liu~\cite{ChenLiu2023}. The mesh remains fixed throughout time
integration, and each physical merged cell carries a single
polynomial in $\mathbb Q_k$. The analysis distinguishes between
two classes of faces. On complete faces shared by congruent
unmerged squares, a Legendre expansion isolates the leading
projection-error contribution. This contribution cancels for
even $k$; for odd $k$, it is canceled by a skew derivative-jump
correction. The remaining faces are confined to a boundary strip
and have uniformly bounded total length. Together, these properties
yield optimal-order weak consistency for the cellwise physical
$L^2$ projection, including on curved cells. The correction makes
no contribution to the discrete energy balance and introduces
no additional unknowns. Cell merging provides uniform inverse
estimates, while the correction supplies the cancellation needed
for optimal accuracy on regular interior faces.

Our main result is the optimal error bound
\[
\max_{0\le t\le T}
\norm{U(t)-u(t)}_{L^2(\Omega)}
\le C_T h^{k+1}
\]
for every fixed $k\ge1$, assuming a smooth scalar solution with
vanishing trace on the entire physical boundary. No sign condition
is imposed on either of the two flux derivatives. We also establish
semidiscrete entropy dissipation and extend
the optimal error estimate to the oscillation-free DG (OFDG)
method~\cite{LuLiuShu2021SINUM}. The analysis assumes exact
integration over the physical geometry, whereas the numerical
experiments employ high-order quadrature with separate refinement
checks. The disk experiment satisfies the boundary assumptions
of the main theorem. The star-shaped-domain experiment tests
an extension to linear transport with an upwind treatment of
inflow boundary conditions, for which we also establish an optimal
semidiscrete error estimate.

The paper is organized as follows.
Section~\ref{sec:geometry} introduces the induced mesh for the
curved computational domain. Section~\ref{sec:method} presents
the merged-cell DG method for scalar conservation laws.
In Section~\ref{sec:spatial}, we establish the superconvergence
properties of the $L^2$ projection on regular faces and prove
$L^2$ stability and an optimal error estimate.
Section~\ref{sec:numerics} presents numerical experiments that
support the theoretical results. Finally,
Section~\ref{sec:concluding} summarizes the main findings and
outlines directions for future work.

\section{Geometry of the merged cells}\label{sec:geometry}

Let $\Omega\Subset B\subset\mathbb R^2$, where $B$ is a rectangle and
$\Gamma=\partial\Omega$ is compact, embedded, and $C^2$.
Start with a uniform Cartesian mesh of side $h$, $\mathcal{T}_h$, in $B$. A Cartesian mesh means the elements of the mesh are rectangles whose sides are parallel to the coordinate axes. The elements intersecting with the boundary are called the boundary elements. From $\Th$ we want to construct an induced mesh $\cM$, which avoids possible small intersections of the boundary with the mesh elements. We denote $\Th^\Gamma:=\{K\in \Th:K\cap \Gamma \neq \emptyset\}$. We recall the definition of large element in Chen et al \cite[Definition 2.1]{CLX2021NM}.

\begin{Def}\label{def:Large-element}
(Large element) An element $K\in \Th$ is called a large element with respect to $\Omega$ if $K\subset\Omega$ or $K\in\Th^\Gamma$ for which there exists a constant $\delta_0\in(0,1/2)$ such that $|e\cap \Omega|\ge \delta_0|e|$ for each side $e$ of $K$ having nonempty intersection with $\Omega$. Otherwise, $K$ is called a small element.
\end{Def}
The rectangular merging construction of Chen and Liu~\cite[Algorithm 6]{ChenLiu2023}
is applied to $\Omega$, on meshes $\mathcal{T}_h$
satisfying the four admissibility rules of \cite[Definition 3.1]{ChenLiu2023}.
Whole background meshes are merged before restriction to $\Omega$;
the exterior subdomain carries no unknowns. We call the merged mesh the induced mesh, which is $\mathcal{M}_h=\{\tilde{K}=M(K)\cap\Omega: M(K) \text{ is the macro-element of $K$}, K\in \Th\}$. Here, we denote $M(K)=K$ if $K$ is unmerged for $K\in \Th$. We denote $\mathcal M_h^\Gamma:=\{\tilde{K}=M(K)\cap\Omega: K\in \Th^\Gamma\}$. We state the
geometric output conditions used in the analysis. The macro-elements $M(K)$
have disjoint interiors and side lengths between $h$ and $C_Mh$.
On each boundary element $\tilde{K}\in \mathcal{M}_h^\Gamma$, the curved boundary is a single simple arc with
exactly two intersections with $\partial M(K)$. Its endpoints lie on
distinct sides of $M(K)$, and the open chord joining them lies in
$\operatorname{int}M(K)$. Assume also that
$\Gamma\cap\partial M(K)=\{B,C\}$, so that there are no additional
boundary contacts. The open line segment $\Gamma_h^K:=(B,C)$ lies in $\interior M(K)$. 
To resolve the geometry of the boundary, we introduce the concept of interface deviation, as introduced in \cite{CLX2021NM}.
\begin{Def}\label{def:2.2}
For any $M_K\in\mathcal{M}^\Gamma$, the interface deviation $\eta_K$ is defined as, if $A_K\in\Omega$ is the vertex of $M_K$ which has the maximum distance to $\Gamma_K^h$ among all vertices of $M_K$ in $\Omega$,
\begin{align}\label{def:interface_deviation}
\eta_K=\frac{\dist_{\rm H}(\Gamma_K,\Gamma_K^h)}{\dist(A_K,\Gamma_K^h)}.    
\end{align}
Here $\dist_{\rm H}(\Gamma_1,\Gamma_2)=\max_{x\in\Gamma_1}(\min_{y\in\Gamma_2}|x-y|)$ and $\dist(A,\Gamma_1)=\min_{y\in\Gamma_2}|A-y|$.
\end{Def}
It is known \cite{CLX2021NM} that $\eta_K\le Ch_K$. Thus, the interface deviation can be arbitrarily small when the mesh is sufficiently refined. The following lemma claims that each large element contains a closed ball of radius $O(h)$, when the interface deviation is sufficiently small.

\begin{lemma}\label{lem:geometry-inball}
For sufficiently small $h$, each physical
cell contains a closed ball of radius $c_bh$, where
\[
 c_b=\frac{\delta_0}{16}.
\]
Moreover, there are constants independent of $h$ and of the
original cut-cell area fractions such that
\begin{equation}\label{eq:conclusions}
 \pi c_b^2h^2\le |K|\le C_M^2h^2,
 \qquad |\partial K|\le Ch.
\end{equation}
More explicitly, the inball conclusion on any cut rectangle holds
whenever
\begin{equation}\label{eq:eta-threshold}
 \eta_K\le\eta_0:=\frac{\delta_0}{32\sqrt2\,C_M}.
\end{equation}
For the fixed compact embedded $C^2$ boundary, this condition is
automatically satisfied by all cut rectangles when $h$ is
sufficiently small.
\end{lemma}

\begin{proof}
An uncut physical rectangle contains a closed ball of radius
$h/4$. Since $\delta_0/16<1/4$, it contains the ball
claimed in the lemma. Its area and perimeter satisfy
$h^2\le |K|\le C_M^2h^2$ and $|\partial K|\le4C_Mh$.
It remains to treat cut rectangles. Let $\Omega_1=\Omega$ and $\Omega_2=\mathbb R^2 \setminus \bar{\Omega}_1$. We prove the inball assertion
for either phase and then take $i=1$.

\medskip\noindent
\textit{Step 1: a uniformly sized triangle in each curved polygon.}
Let $M:=M(K)$ for $K\in \Th^\Gamma$. The supporting line $\Gamma_K^h$ divides $M(K)$ into two closed
convex polygons. Denote them by $P_{M,1}$ and $P_{M,2}$, assigning
the labels by the corresponding phase on each of the two
rectangle-perimeter paths from $B$ to $C$. This labeling is
well defined: each open perimeter path misses $\Gamma$, is
connected, and therefore lies in one phase. The simple arc
$\Gamma_K$ is a crosscut of the rectangle, so the two paths
have opposite phase labels.

For a fixed $i$, let $\sigma_i$ be signed distance to $\Gamma_K^h$,
positive in $P_{M,i}$. Thus
\begin{equation}\label{eq:signed}
 P_{M,i}=M\cap\{\sigma_i\ge0\},\qquad
 |\nabla\sigma_i|=1.
\end{equation}
Choose a rectangle vertex $A$ maximizing
$\sigma_i$ over $M$, and write $d=\sigma_i(A)>0$.
This vertex is not $B$ or $C$ and belongs to the phase-$i$
perimeter path; hence $A\in P_{M,i}$.

Let $e_1,e_2$ be the inward unit coordinate directions along the
two rectangle sides issuing from $A$, and set $\ell=\delta_0h$. There is at most one interface endpoint in the
relative interior of either side. Consequently its phase-$i$
portion starting at $A$ has length at least $\delta_0$ times
that side length, by Definition \ref{def:Large-element}, and hence at least
$\ell$. If the side has no interface endpoint, the entire side
belongs to the same perimeter path, apart from possible endpoints.
The corresponding curved-polygon side portions therefore contain
$A+\ell e_1$ and $A+\ell e_2$. Convexity gives
\begin{equation}\label{eq:triangle}
 T_A:=\conv\{A,A+\ell e_1,A+\ell e_2\}
 \subset P_{M,i}.
\end{equation}

Define
\begin{equation}\label{eq:center}
 c=A+\frac{\ell}{4}(e_1+e_2),\qquad r_0=\frac{\ell}{8}.
\end{equation}
In coordinates $x=A+t_1e_1+t_2e_2$, the triangle is
$t_1\ge0$, $t_2\ge0$, $t_1+t_2\le\ell$.
The distances from $c$ to its three supporting lines are
\[
 \frac{\ell}{4},\qquad \frac{\ell}{4},\qquad
 \frac{\ell-\ell/4-\ell/4}{\sqrt2}
 =\frac{\ell}{2\sqrt2}.
\]
All exceed $r_0$. Since the rectangle side lengths are at least
$h\ge\ell$, we conclude that
\begin{equation}\label{eq:polygon-ball}
 \overline B(c,r_0)\subset T_A\cap\interior M,
 \qquad \sigma_i(c)\ge r_0.
\end{equation}
The last inequality follows because the ball lies in the
half-plane $\{\sigma_i\ge0\}$ and $\sigma_i$ is signed distance.

We also record a lower bound for the denominator in
\eqref{def:interface_deviation}. Maximality of $A$ implies
\[
 \sigma_i(A+t_1e_1+t_2e_2)=d-\alpha_1t_1-\alpha_2t_2,
 \qquad \alpha_1,\alpha_2\ge0,\qquad
 \alpha_1^2+\alpha_2^2=1.
\]
By \eqref{eq:triangle}, $d-\ell\alpha_j\ge0$ for $j=1,2$.
Therefore
\begin{equation}\label{eq:denominator}
 \frac{\delta_0h}{\sqrt2}
 \le \ell\max(\alpha_1,\alpha_2)
 \le d=\dist(A,\Gamma_K^h)
 \le\diam M\le\sqrt2\,C_Mh.
\end{equation}
In particular the deviation is well defined. Notice that only an
inequality, not equality, is used between the line and segment
distances.

\medskip\noindent
\textit{Step 2: interface deviation gives an inner polygon in the true phase.}
By \eqref{def:interface_deviation} and the upper bound in
\eqref{eq:denominator}, condition \eqref{eq:eta-threshold} gives
\begin{equation}\label{eq:delta-small}
 \delta_K=\eta_{K} \dist(A,\Gamma_K^h)
 \le\eta_K\sqrt2\,C_Mh
 \le\frac{\delta_0h}{32}=\frac{\ell}{32}.
\end{equation}
For every $x\in\Gamma_K$, \eqref{def:interface_deviation} yields
\begin{equation}\label{eq:strip}
 |\sigma_i(x)|=\dist(x,\Gamma_K^h)\le\delta_K.
\end{equation}
Consider the open inward-shifted polygon
\begin{equation}\label{eq:inner-polygon}
 U_{M,i}=\interior M\cap\{x:\sigma_i(x)>\delta_K\}.
\end{equation}
The set $U_{M,i}$ is convex. It misses $\Gamma$ by
\eqref{eq:strip}. Moreover, \eqref{eq:denominator} and
\eqref{eq:delta-small} give $\sigma_i(A)>\delta_K$.
Since $A\in\Omega_i$ and $\Omega_i$ is open, for all sufficiently
small $t>0$ the point
\[
 y_t=A+t(e_1+e_2)
\]
belongs to $U_{M,i}$.
For any $x\in U_{M,i}$, the segment from $y_t$ to $x$ stays in
$U_{M,i}$ and cannot meet $\Gamma$.
The sets $\Omega_1$ and $\Omega_2$ are disjoint open sets whose
union is $\R^2\setminus\Gamma$, so a connected segment avoiding
$\Gamma$ cannot change phase. Thus
\begin{equation}\label{eq:inner-inclusion}
 U_{M,i}\subset\Omega_i.
\end{equation}

\medskip\noindent
\textit{Step 3: the physical inball.}
If $x\in\overline B(c,r_0/2)$, then
\eqref{eq:polygon-ball} and the unit Lipschitz constant of
$\sigma_i$ imply
\begin{align*}
 \sigma_i(x)
 &\ge\sigma_i(c)-|x-c|\ge r_0-r_0/2
   =\frac{\ell}{16}>\frac{\ell}{32}\ge\delta_K.
\end{align*}
Also $x\in\interior M$. Combining this with
\eqref{eq:inner-inclusion} proves
\begin{equation}\label{eq:physical-ball}
 \overline B\left(c,\frac{\delta_0h}{16}\right)
 \subset U_{M,i}\subset\interior M\cap\Omega_i.
\end{equation}
This proves the claimed radius without assuming that the physical
cell $\tilde{K}$ is convex or star-shaped.

\medskip\noindent
\textit{Step 4: area and boundary length.}
For $i=1$, \eqref{eq:physical-ball} and $K\subset M$ give
\[
 \pi\left(\frac{\delta_0}{16}\right)^2h^2
 \le |\tilde{K}|\le |M|\le C_M^2h^2.
\]
We assume the arc $\Gamma_K$ has the parameter representation
\[
 \Gamma_K=\{(t,q(t)):t_0\le t\le t_1\},\qquad t_0<t_1.
\] 
The interface arclength obeys
\begin{align*}
 |\Gamma_K|
 &=\int_{t_0}^{t_1}\sqrt{1+|q'(t)|^2}\,dt
 \le\sqrt2(t_1-t_0)\le\sqrt2\,\diam M\le2C_Mh.
\end{align*}
Since $\partial \tilde{K}\subset\partial M\cup \Gamma_K$, up to the
endpoints of zero length,
\[
 |\partial \tilde{K}|\le|\partial M|+|\Gamma_K|
 \le4C_Mh+2C_Mh=6C_Mh.
\]
This proves \eqref{eq:conclusions} and completes the proof.
\end{proof}

Constants of all estimates below may depend on the fixed
geometry, $\delta_0$, $C_M$, $\eta_0$, and the fixed polynomial degree, but not
on the position or area fraction of an original cut cell. The following curved domain inverse estimate is proved in \cite[Lemma 2.4]{CLX2021NM}
\begin{lemma}\label{lem:geometry-inverse}
For $p\in\mathbb Q_k(M(K))$ and a multi-index $\alpha$,
\begin{equation}\label{eq:geometry-inverse-maximum}
 \|\partial^\alpha p\|_{L^\infty(\tilde{K})}
 \le C_{k,\alpha}h^{-1-|\alpha|}\|p\|_{L^2(\tilde{K})},
\end{equation}
where $\tilde{K}=M(K)\cap \Omega$.
In particular,
\begin{equation}\label{eq:geometry-inverse}
 \|\nabla p\|_{L^2(\tilde{K})}\le Ch^{-1}\|p\|_{L^2(\tilde{K})},\qquad
 \|p\|_{L^2(\partial \tilde{K})}\le Ch^{-1/2}\|p\|_{L^2(\tilde{K})}.
\end{equation}
\end{lemma}

Let $E_h=\bigcup_{K\in\mathcal M_h^\Gamma}K$ be the exceptional region. The regular faces $\Fh^r$
are full sides shared by two original uncut, unmerged squares;
the other interior exceptional faces form $\Fh^e$ and physical boundary faces form $\Fh^b$. Let $\Fh$ denote all interior and physical boundary
faces, that is $\Fh=\Fh^r\cup \Fh^e \cup \Fh^b$.

\begin{lemma}\label{lem:geometry-strip}
The exceptional region and its faces satisfy
\begin{equation}\label{eq:geometry-strip}
 |E_h|\le Ch,\qquad \#\mathcal M_h^\Gamma\le Ch^{-1},\qquad
 \sum_{F\in\Fh^e}|F|+|\Gamma|\le C.
\end{equation}
Furthermore $\sum_{F\in\Fh}|F|\le Ch^{-1}$ and, for every function $v$
whose restriction to each cell belongs to $\mathbb Q_k$,
\begin{equation}\label{eq:geometry-summed-trace}
 \sum_{K\in\cM}\|v\|_{L^2(\partial K)}^2
 \le Ch^{-1}\|v\|^2.
\end{equation}
The same conclusions hold after marking any fixed number of
additional neighboring layers as exceptional.
\end{lemma}
\begin{proof}
Local merging and the diameter bound place $E_h$ in a tube of
width $Rh$ about $\Gamma$, with fixed $R$. Compact embedded
$C^2$ regularity gives an injective normal parametrization in a
fixed tube; its Jacobian in arclength coordinates is
$1-t\kappa(s)$. Hence this tube has area at most $Ch$.
Disjoint cell interiors and the lower area bound yield
$\#\mathcal M_h^\Gamma\le Ch^{-1}$. Summing cell perimeters gives the
exceptional-face estimate. Similarly, $\#\mathcal{M}_h\le Ch^{-2}$ and
$|\partial K|\le Ch$ bound the whole skeleton. Summing
\eqref{eq:geometry-inverse} proves
\eqref{eq:geometry-summed-trace}. Face subdivision partitions
arclength and does not multiply any trace integral; an interior
face occurs twice and a physical boundary face once in the
cell-boundary sum. Adding finitely many neighboring layers only
increases the fixed tube width.
\end{proof}
\section{The method and the main result}\label{sec:method}
We consider the scalar conservation law on $\Omega$ with curved $C^2$ boundary,
\begin{equation}\label{eq:pde}
 u_t+\nabla\cdot \FF(u)=0,\qquad \FF=(f,g),\qquad u(0)=u_0.
\end{equation}
We consider the homegeous boundary condtion,
\begin{equation}\label{eq:boundary}
 u|_{\partial\Omega}=0\quad(0\le t\le T)
\end{equation}
This is a compatibility condition on a classical solution, stronger
than an inflow prescription. It does not require that $u$ vanish in
a boundary neighborhood. We assume a $C^{k+7}$ space--time extension
of $u$ to a fixed neighborhood of $\overline\Omega\times[0,T]$ and
$\FF\in C^{k+7}(\R;\R^2)$. All results concern this smooth time interval.
The degree $k\ge1$ is fixed. Write $\ip{v}{w}=(v,w)_{L^2(\Omega)}$, $\norm{v}=\norm{v}_{L^2(\Omega)}$,
and $\norm{v}_E=\norm{v}_{L^2(E)}$. On the mesh of
Section~\ref{sec:geometry}, define
\[
 \Vh^k=\{v\in L^2(\Omega):v|_{\tilde{K}}\in\mathbb Q_k(M(K))|_{\tilde{K}},
              \ \tilde{K}=M(K)\cap \Omega \in\mathcal M_h \text{ for some } K \in \Th\}.
\]
Here $\mathbb Q_k$ has degree at most $k$ in each Cartesian coordinate.
The projection $\Ph$ is $ L^2$-orthogonal in the actual cell $\tilde{K}\in \cM$, rather than in its containing rectangle, that is $P_hw\in V_h^k$ such that
\begin{align}\label{L2-projection}
    \int_{\tilde{K}}(P_hw-w)v\, dx=0\quad \forall v\in \mathbb Q_k(M(K)),\, \tilde{K}\in \cM.
\end{align}

\subsection{Fluxes and the skew correction}
For any $F\in \mathcal F_h$, we fix a unit normal vector $\mathbf{n}_F$ of $F$ with the convention that $\mathbf{n}_F$ is the unit outer normal to $\partial \Omega$ if $F\subset \Gamma$, and
write $\jp{v}=v^- -v^+$ and $\{v\}=(v^-+v^+)/2$, where $v^{\pm}(\mathbf{x}):=\lim_{\varepsilon\rightarrow 0^+}v(\mathbf{x}\pm\varepsilon \mathbf{n}_F)$. For a fixed $\nu_0>0$, set $\nu_F=\nu_0h$ on $\Fh^r$ and
$\nu_F=\nu_0$ elsewhere, including on the physical boundary. Define
\begin{align}
 \FF^{\rm ec}(a,b)&=\int_0^1 \FF((1-s)a+sb)\,ds,
       \label{eq:ec-flux}\\
 \widehat \F_{\mathbf{n}_F,F}(a,b)&=\FF^{\rm ec}(a,b)\cdot \mathbf{n}_F
                         +\tfrac12\nu_F(a-b).
       \label{eq:num-flux}
\end{align}
The state integral and all physical volume and face integrals are exact
in the analysis.
Put $c_k=(1+(-1)^{k+1})k!/[4(2k+1)!]$ and
\begin{equation}\label{eq:skew}
 \begin{split}
 \Sh(w,v)&=c_kh^k\sum_{F\in\Fh^r}\int_F a_F(w)
 \bigl(\jp{\partial_{\mathbf{n}_F}^kw}\jp{v}-\jp{w}\jp{\partial_{\mathbf{n}_F}^kv}\bigr),\\
 a_F(w)&=\FF'(\{w\})\cdot \mathbf{n}_F.
 \end{split}
\end{equation}
Note that $\Sh=0$ for even $k$. Both derivative traces use the same oriented
normal $\mathbf{n}_F$. For odd $k$, exchanging the cells and reversing the
normal leaves \eqref{eq:skew} unchanged. Indeed, $\jp{w}$ and $a_F$ change
sign, whereas $\jp{\partial_{\mathbf{n}_F}^kw}$ does not. For every $w$,
$\Sh(w,w)=0$.

The spatial DG operator $\Lh:\Vh^k\to\Vh^k$ is defined by
\begin{equation}\label{eq:operator}
 \begin{split}
 \ip{\Lh(w)}{v}={}&\sum_{K\in\cM}(\FF(w),\nabla v)_K
 -\sum_{\F\in\Fh^r\cup\Fh^e}\int_F
               \widehat \F_{\mathbf{n}_F,F}(w^-,w^+)\jp{v}\\
 &-\sum_{F\in \Fh^b}\int_{F}\widehat \F_{\mathbf{n}_F,F}(w,0)v+\Sh(w,v).
 \end{split}
\end{equation}
The semidiscrete method is find $U\in V_h$, such that $U_t=\Lh(U)$, $U(0)=\Ph u_0$.
The mass matrices are positive definite by
Lemma~\ref{lem:geometry-inball}. Define
\begin{equation}\label{eq:dissipation}
 \Dh(v)=\sum_{F\in\Fh^r\cup\Fh^e}\nu_F\norm{\jp{v}}_F^2
                         +\nu_0\norm{v}_{\partial\Omega}^2.
\end{equation}
The particular scaling on regular faces balances their corrected
projection residual with $\Dh$.

The method is locally conservative. Testing with the cell indicator
$\chi_K$ makes its volume gradient and derivative jumps vanish.
On an odd-degree regular face the effective oriented flux is
\[
  H_F(w)=\widehat \F_{\mathbf{n}_F,F}(w^-,w^+)
                 -c_kh^ka_F(w)\jp{\partial_{\mathbf{n}_F}^kw};
\]
on the other faces it is $\widehat F_{\mathbf{n}_F,F}$. If
$\epsilon_{K,F}=1$ for the minus cell and $-1$ for the plus cell, then
\[
 \frac{d}{dt}\int_K U
 =-\sum_{F\subset\partial K\setminus\partial\Omega}
                  \epsilon_{K,F}\int_FH_F(U)
   -\int_{\partial K\cap\partial\Omega}\widehat F_{\mathbf{n}_F,F}(U,0).
\]
Interior contributions cancel in the sum over cells. Both jumps in
$\Sh$ vanish for a common smooth exact state, so the correction is
consistent as well as conservative.

\subsection{Error and stability}
The following Theorem is our main result; we have $L^2$-stability and an optimal a priori error estimate for the semi-discrete scheme.

\begin{theorem}\label{thm:main}
Assume the mesh conditions in Section~\ref{sec:geometry} and the
regularity and boundary conditions above. For every fixed $k\ge1$,
the semidiscrete solution initialized by $\Ph u_0$ satisfies
\begin{align}
 &\norm{U(t)}^2+\int_0^t\Dh(U(s))\,ds=\norm{\Ph u_0}^2,
 \label{eq:semidiscrete-energy}\\
 &\sup_{0\leq t\leq T}\norm{U(t)-u(t)}\leq C_T h^{k+1},
 \qquad \int_0^T\Dh(U-\Ph u)\,dt\leq C_T h^{2k+2}.
 \label{eq:semidiscrete-error}
\end{align}
\end{theorem}

\section{Projection and spatial error analysis}\label{sec:spatial}
All projections and inner products in this section use the physical
cells. We write $\epsh=h^{k+1}$ and restrict $h\leq1$. Constants are
uniform on bounded sets of the smooth extensions in Theorem~\ref{thm:main}.
For estimates involving arbitrary discrete errors, the flux is first
smoothly extended outside a fixed interval containing the exact solution
range, so that its first three derivatives are bounded. The extension
is removed at the end of the error argument.

\begin{lemma}[Physical-cell projection]\label{lem:projection}
Let $W=\Ph\phi$ and $\eta=\phi-W$, where $\phi$ has a bounded
$W^{k+2,\infty}$ extension to the macro-element $M(K)$ for some $K\in \Th$. Let $\tilde{K}=M(K)\cap\Omega$, then
\begin{equation}\label{eq:projection-bounds}
 \norm{\eta}+\max_{K\in \Th}\norm{\eta}_{L^\infty(\tilde{K})}
 \leq C\epsh,\qquad
 \max_{K\in \Th}\norm{\nabla W}_{L^\infty(\tilde{K})}\leq C.
\end{equation}
In particular, $\jp{W}=O(\epsh)$ on every interior face. If
$\phi|_{\partial\Omega}=0$, then $W=O(\epsh)$ on the physical boundary.
\end{lemma}
\begin{proof}
On $\tilde{K}$, choose the total-degree-$k$ Taylor polynomial $p$
of $\phi$ at the center of $M(K)$. It belongs to $\mathbb Q_k(M(K))$, and
$r=\phi-p$ satisfies
$\norm{r}_{L^\infty(M(K))}\leq Ch^{k+1}$ and
$\norm{\nabla r}_{L^\infty(M(K))}\leq Ch^k$.
The inverse estimate of Lemma~\ref{lem:geometry-inverse} and
$L^2(\tilde{K})$ contractivity give, for any bounded $z$,
\[
 \norm{\Ph z}_{L^\infty(\tilde{K})}
 \leq Ch^{-1}\norm{\Ph z}_{\tilde{K}}
 \leq Ch^{-1}|\tilde{K}|^{1/2}\norm{z}_{L^\infty(\tilde{K})}
 \leq C\norm{z}_{L^\infty(\tilde{K})}.
\]
Its differentiated version gives
$\norm{\nabla\Ph z}_{L^\infty(\tilde{K})}\leq Ch^{-1}\norm{z}_{L^\infty(\tilde{K})}$.
Polynomial reproduction therefore yields
$W=p+\Ph r$, $\eta=r-\Ph r$, and the pointwise bounds in
\eqref{eq:projection-bounds}. Summing
$\norm{\eta}_{\tilde{K}}^2\leq C\epsh^2|\tilde{K}|$ proves its $L^2$ bound.
Continuity of the exact trace gives $\jp{W}=-\jp{\eta}$; the boundary
claim follows from $W=-\eta$ there.
\end{proof}
The cancellation below follows from the parity of Legendre polynomials. Related cancellation arguments on uniform meshes appear in Liu, Shu, and Zhang \cite{LiuShuZhang2021Central}. We give a direct proof for the cellwise \(L^2\) projection, including the derivative-jump correction for odd degrees.
\begin{lemma}[Cancellation on regular faces]\label{lem:parity}
Let $F\in\Fh^r$ and let $P_F$ be its tangential $L^2$ projection onto
$\mathbb P_k(F)$. For $\eta=\phi-\Ph\phi$ as above, even degrees satisfy
\begin{equation}\label{eq:parity-even}
 \norm{P_F\{\eta\}}_{L^\infty(F)}\leq Ch^{k+2}.
\end{equation}
For odd degrees,
\begin{align}
 \norm{\jp{\eta}}_{L^\infty(F)}&\leq Ch^{k+2},\label{eq:parity-jump}\\
 \norm{P_F\{\eta\}-c_kh^k\jp{\partial_{\mathbf{n}_F}^k\eta}}_{L^\infty(F)}
 &\leq Ch^{k+2},\qquad
 h^k\norm{\jp{\partial_{\mathbf{n}_F}^k\eta}}_{L^\infty(F)}\leq C\epsh.
 \label{eq:parity-corrected}
\end{align}
\end{lemma}
\begin{proof}
\textit{1. Explicit projectors and their bounds.}
After translation, consider
\[
 \begin{gathered}
 I_-=(-h,0),\qquad I_+=(0,h),\qquad J=(-h/2,h/2),\\
 K^\pm=I_\pm\times J,\qquad F=\{0\}\times J,\qquad n_F=(1,0).
 \end{gathered}
\]
Put $m=k+1$ and introduce the reference coordinates
\[
 \xi_-(x)=1+\frac{2x}{h},\qquad
 \xi_+(x)=-1+\frac{2x}{h},\qquad \zeta(y)=\frac{2y}{h}.
\]
Let $\ell_j$ be the Legendre polynomial
\[
 \ell_j(t)=\frac{1}{2^jj!}\frac{d^j}{dt^j}(t^2-1)^j.
\]
Its normalization, orthogonality, parity, and leading coefficient are
\begin{equation}\label{eq:parity-legendre-data}
 \int_{-1}^1\ell_i(t)\ell_j(t)\,dt
       =\frac{2\delta_{ij}}{2j+1},\quad
 \ell_j(1)=1,\quad \ell_j(-1)=(-1)^j,\quad
 \ell_j(t)=2^{-j}\binom{2j}{j}t^j+p_{j-2}(t),
\end{equation}
where $p_{j-2}\in\mathbb P_{j-2}$; a negative-degree term is zero.

For $\sigma\in\{-,+\}$ define the normal projector $P_\sigma$ by
\begin{equation}\label{eq:parity-explicit-normal-projector}
 (P_\sigma z)(x,y)
 =\sum_{j=0}^k\frac{2j+1}{h}
   \left(\int_{I_\sigma}z(t,y)\ell_j(\xi_\sigma(t))\,dt\right)
                       \ell_j(\xi_\sigma(x)).
\end{equation}
Define the tangential projector $Q$ by
\begin{equation}\label{eq:parity-explicit-tangent-projector}
 (Qz)(x,y)
 =\sum_{j=0}^k\frac{2j+1}{h}
   \left(\int_Jz(x,s)\ell_j(\zeta(s))\,ds\right)\ell_j(\zeta(y)).
\end{equation}
On the face, $Q=P_F$. Fubini's theorem and the product orthogonal
basis give the exact identities
\begin{equation}\label{eq:parity-projector-factorization}
 P_h|_{K^\sigma}=QP_\sigma=P_\sigma Q,
 \qquad \partial_x^qQz=Q\partial_x^qz
       \quad(z\in W^{q,\infty}(K^\sigma)).
\end{equation}
No commutation between $\partial_x$ and $P_\sigma$ is used.
Set
\[
 A_{k,q}=2^q\sum_{j=0}^k(2j+1)
      \|\ell_j\|_{L^\infty(-1,1)}
      \|\ell_j^{(q)}\|_{L^\infty(-1,1)},\qquad 0\le q\le k.
\]
The explicit sums \eqref{eq:parity-explicit-normal-projector}--
\eqref{eq:parity-explicit-tangent-projector} imply
\begin{equation}\label{eq:parity-projector-bounds}
 \|\partial_x^qP_\sigma z\|_{L^\infty(K^\sigma)}
      \le A_{k,q}h^{-q}\|z\|_{L^\infty(K^\sigma)},\qquad
 \|Qz\|_\infty\le A_{k,0}\|z\|_\infty.
\end{equation}
These inequalities also bound evaluations at $x=0$, because
$P_\sigma z$ is a polynomial in $x$.

\textit{2. One common normal Taylor expansion.}
Define $f_j(y)=\partial_x^j\phi(0,y)$ for $0\le j\le m$ and write
\begin{align}
 \phi(x,y)&=T_m(x,y)+R(x,y),\qquad
 T_m(x,y)=\sum_{j=0}^m\frac{f_j(y)}{j!}x^j,
                                      \label{eq:parity-taylor}\\
 R(x,y)&=\frac{x^{m+1}}{m!}\int_0^1(1-t)^m
                      \partial_x^{m+1}\phi(tx,y)\,dt.
                                      \label{eq:parity-taylor-remainder}
\end{align}
The same coefficients $f_j$ and the same remainder function $R$
are used on both sides of the face. Taylor's formula yields
\begin{equation}\label{eq:parity-remainder-bounds}
 \|R\|_{L^\infty(K^-\cup K^+)}
       \le C_{\phi}\frac{h^{m+1}}{(m+1)!},
 \qquad \partial_x^qR(0,y)=0\quad(0\le q\le m),
\end{equation}
where $C_\phi=\|\phi\|_{W^{k+2,\infty}(K^-\cup K^+)}$. In particular, the derivative remainder at the face vanishes exactly
before any projector is applied.

Put
\[
 B_m=\binom{2m}{m},\qquad
 d_h(y)=\frac{h^m}{m!B_m}f_m(y),\qquad D_h=Qd_h.
\]
The leading coefficient in \eqref{eq:parity-legendre-data} and
orthogonality give
\begin{equation}\label{eq:parity-monomial-projection}
 (I-P_\sigma)x^j=0\quad(0\le j\le k),\qquad
 (I-P_\sigma)x^m=\frac{h^m}{B_m}\ell_m(\xi_\sigma(x)).
\end{equation}
Indeed, $x^m-h^mB_m^{-1}\ell_m(\xi_\sigma(x))$ has degree at most
$m-1=k$, and $P_\sigma\ell_m(\xi_\sigma)=0$.
Consequently,
\begin{equation}\label{eq:parity-normal-error-exact}
 (I-P_\sigma)\phi
   =d_h\ell_m(\xi_\sigma)+(I-P_\sigma)R.
\end{equation}
Using $I-QP_\sigma=(I-Q)+Q(I-P_\sigma)$ gives the exact cell identity
\begin{equation}\label{eq:parity-tensor-error-exact}
 \eta|_{K^\sigma}
  =(I-Q)\phi+D_h\ell_m(\xi_\sigma)+Q(I-P_\sigma)R.
\end{equation}

\textit{3. Exact trace identities with quantified remainders.}
Define the face functions
\[
 g_h=(I-Q)f_0,\qquad
 r_\sigma=-Q(P_\sigma R)(0,\cdot),\qquad
 s_\sigma=-Q(\partial_x^kP_\sigma R)(0,\cdot).
\]
Here $Qr_\sigma=r_\sigma$, $Qs_\sigma=s_\sigma$, and $Qg_h=0$.
Let
\[
 C_r=\frac{A_{k,0}^2}{(m+1)!},\qquad
 C_s=\frac{A_{k,0}A_{k,k}}{(m+1)!}.
\]
Equations~\eqref{eq:parity-projector-bounds} and
\eqref{eq:parity-remainder-bounds} give
\begin{equation}\label{eq:parity-face-remainder-bounds}
 \|r_\sigma\|_{L^\infty(F)}\le C_r h^{m+1}C_\phi,
 \qquad
 \|s_\sigma\|_{L^\infty(F)}\le C_s h^{m+1-k}C_\phi=C_sh^2C_\phi.
\end{equation}
Since $\xi_-(0)=1$, $\xi_+(0)=-1$ and $R(0,\cdot)=0$,
\eqref{eq:parity-tensor-error-exact} yields
\begin{equation}\label{eq:parity-value-traces}
 \eta^-=g_h+D_h+r_-,\qquad
 \eta^+=g_h+(-1)^mD_h+r_+.
\end{equation}
Taking a projected average and a jump, respectively, gives
\begin{align}
 P_F\{\eta\}
   &=\frac{1+(-1)^m}{2}D_h+\frac{r_-+r_+}{2},
                                      \label{eq:parity-average-identity}\\
 \jp{\eta}&=(1-(-1)^m)D_h+r_--r_+.
                                      \label{eq:parity-jump-identity}
\end{align}
The tangential term $g_h$ disappears from the first equation because
$Qg_h=0$ and from the second because it is the same on both sides.

For the derivative traces, \eqref{eq:parity-legendre-data} implies
\[
 \ell_m^{(m-1)}(t)=2^{-m}B_m m!\,t,
\]
and hence, since $k=m-1$,
\begin{equation}\label{eq:parity-leading-derivative}
 \partial_x^k\left(\frac{h^m}{m!B_m}
                         \ell_m(\xi_\sigma(x))\right)
 =\frac{h^m}{m!B_m}\left(\frac2h\right)^{m-1}
                         2^{-m}B_m m!\,\xi_\sigma(x)
 =\frac h2\xi_\sigma(x).
\end{equation}
Differentiate \eqref{eq:parity-tensor-error-exact} $k$ times,
use \eqref{eq:parity-projector-factorization} and
$\partial_x^kR(0,\cdot)=0$, and set $b_h=(I-Q)f_k$. This gives
\begin{equation}\label{eq:parity-derivative-traces}
 (\partial_x^k\eta)^-=b_h+\frac h2Qf_m+s_-,\qquad
 (\partial_x^k\eta)^+=b_h-\frac h2Qf_m+s_+.
\end{equation}
In particular,
\begin{equation}\label{eq:parity-derivative-jump-identity}
 \jp{\partial_x^k\eta}=hQf_m+\rho_h\in\mathbb P_k(F),\qquad
 \rho_h=s_--s_+,\quad \|\rho_h\|_{L^\infty(F)}\le2C_sh^2C_\phi.
\end{equation}

\textit{4. The two parity cases.}
If $k$ is even, $m=k+1$ is odd, so
\eqref{eq:parity-average-identity} becomes
\[
 P_F\{\eta\}=\frac{r_-+r_+}{2},\qquad
 \|P_F\{\eta\}\|_\infty
     \le\frac12(\|r_-\|_\infty+\|r_+\|_\infty)
     \le C_rh^{k+2}C_\phi.
\]
This is \eqref{eq:parity-even}.

If $k$ is odd, $m$ is even. First,
\eqref{eq:parity-jump-identity} gives
\[
 \jp{\eta}=r_--r_+,\qquad
 \|\jp{\eta}\|_\infty\le2C_rh^{k+2}C_\phi,
\]
which is \eqref{eq:parity-jump}. Next the exact factorial identity
\begin{equation}\label{eq:parity-coefficient-identity}
 \frac1{m!B_m}=\frac{m!}{(2m)!}
       =\frac{k!}{2(2k+1)!}=c_k
 \qquad(m=k+1)
\end{equation}
implies $D_h=c_kh^{k+1}Qf_m$. Subtracting $c_kh^k$ times
\eqref{eq:parity-derivative-jump-identity} from
\eqref{eq:parity-average-identity} therefore gives
\begin{equation}\label{eq:parity-exact-corrected-remainder}
 P_F\{\eta\}-c_kh^k\jp{\partial_x^k\eta}
       =\frac{r_-+r_+}{2}-c_kh^k\rho_h.
\end{equation}
Consequently,
\[
 \|P_F\{\eta\}-c_kh^k\jp{\partial_x^k\eta}\|_\infty
 \le C_rh^{k+2}C_\phi+2c_kC_sh^{k+2}C_\phi
       =(C_r+2c_kC_s)h^{k+2}C_\phi,
\]
proving \eqref{eq:parity-corrected}. Finally,
\[
 h^k\|\jp{\partial_x^k\eta}\|_\infty
 \le h^{k+1}\|Qf_m\|_\infty+2C_sh^{k+2}C_\phi
 \le(A_{k,0}+2C_s)h^{k+1}C_\phi,
\]
because $h\le1$ and $\|f_m\|_\infty\le C_\phi$.
This proves \eqref{eq:parity-corrected}.

A horizontal face follows by exchanging $x$ and $y$. If the normal
is reversed and the trace labels are exchanged, then
\[
 \jp{\eta}_{\rm new}=-\jp{\eta}_{\rm old},\qquad
 \jp{\partial_{-\mathbf n_F}^k\eta}_{\rm new}
           =(-1)^{k+1}\jp{\partial_{\mathbf n_F}^k\eta}_{\rm old}.
\]
For odd $k$ the second expression is unchanged, while the projected
average is always unchanged. Thus the estimates hold for either
choice of the fixed face orientation.

For the stated Sobolev regularity, the common extension has a
$C^{m,1}$ representative on this rectangle. The normal Taylor formula
\eqref{eq:parity-taylor-remainder} and its zero-jet identities hold
for almost every $y$ by the one-dimensional Sobolev fundamental theorem
of calculus. The highest derivative $\partial_x^{m+1}\phi$ is used
only under the integral and needs no face trace. All bounds above
are essential-supremum bounds, so the same argument proves the lemma
under exactly its stated $W^{k+2,\infty}$ hypothesis.
\end{proof}
\begin{lemma}[Weak residual]\label{lem:weak-residual}
For a smooth zero-trace function $\phi$ satisfying the extension bounds,
put $W=P_h\phi$, $\eta=\phi-W$ and
\[
 \mathcal A(\phi)=-\nabla\cdot\FF(\phi),\qquad
 r(\phi)=P_h\mathcal A(\phi)-\mathcal L_h(W).
\]
For every $v\in V_h$,
\begin{equation}\label{eq:weak-residual}
 |(r(\phi),v)|\le C\varepsilon_h
       \bigl(\|v\|+\Dh(v)^{1/2}\bigr),\qquad
 \|r(\phi)\|\le C\varepsilon_h h^{-1/2}.
\end{equation}
The constants are uniform on bounded sets of the stated smooth functions.
\end{lemma}
\begin{proof}
We may take $0<h\le1$. Write $\Gamma=\partial\Omega$ and use the same
oriented normal on both traces of each interior face. Orthogonality of
$P_h$ and integration by parts on the actual physical cells give the
following exact identity, including the signs of the penalty terms:
\begin{align}
 (r(\phi),v)
 &=\sum_{K\in \cM}(\FF(\phi)-\FF(W),\nabla v)_K
   -\sum_{F\in\mathcal F_h^r\cup\mathcal F_h^e}\int_F
       \bigl(\FF(\phi)\cdot \mathbf{n}_F-\widehat F_{\mathbf{n}_F,F}(W^-,W^+)\bigr)\jp{v}
       \notag\\
 &\quad-\sum_{F\in \mathcal{F}_h^b}\int_F
       \bigl(\FF(0)\cdot n-\widehat F_{\mathbf{n}_F,F}(W,0)\bigr)v
       -S_h(W,v).\label{eq:residual-exact-expanded}
\end{align}
Here and below a sum over $K$ means the physical merged cells.

\paragraph{\bf Volume terms.}
Choose one point $z_K$ in the parent rectangle and set
$\mathbf b_K=\FF'(\phi(z_K))$. Taylor's formula in integral form gives
\begin{align*}
 \FF(\phi)-\FF(W)&=\mathbf b_K\eta+\mathbf R_K,\\
 \mathbf R_K&=\int_0^1
  \bigl(\FF'(\phi-(1-t)\eta)-\mathbf b_K\bigr)\eta\,dt.
\end{align*}
The extension has a uniformly bounded gradient; the macro-element diameter is
at most $Ch$, and $\|\eta\|_{L^\infty(K)}\le C\varepsilon_h$.
Consequently
\[
 |\mathbf R_K|\le C(h+|\eta|)|\eta|\le Ch\varepsilon_h,
 \qquad
 \left(\sum_{K\in \cM}\|\mathbf R_K\|_K^2\right)^{1/2}\le Ch\varepsilon_h.
\]
Each component of $\nabla v$ lies in $\mathbb Q_k$ on the macro-element. Thus $(\mathbf b_K\eta,\nabla v)_K=0$, even when $K$ is curved.
Cauchy--Schwarz and the gradient inverse inequality now give
\begin{equation}\label{eq:residual-volume-expanded}
 \left|\sum_{K\in \cM}(\F(\phi)-\F(W),\nabla v)_K\right|
 \le Ch\varepsilon_h\left(\sum_{K\in \cM}\|\nabla v\|_K^2\right)^{1/2}
 \le C\varepsilon_h\|v\|.
\end{equation}

\paragraph{\bf Regular faces and the parity cancellation.}
On a regular face set $\eta_s=(1-s)\eta^-+s\eta^+$ and
$a_0=\FF'(\phi(z_F))\cdot \mathbf n_F$, where $z_F$ is the face midpoint.
Since the exact trace is single-valued, the entropy-conservative
part of the flux satisfies
\begin{align*}
 \FF(\phi)-\FF^{\rm ec}(W^-,W^+)
 &=\FF'(\phi)\{\eta\}-\mathbf T_F,\\
 \mathbf T_F&=\int_0^1\int_0^1(1-t)
       \FF''(\phi-t\eta_s)\eta_s^2\,dt\,ds.
\end{align*}
Define
\[
 E_F=\bigl(\FF'(\phi)-\FF'(\phi(z_F))\bigr)\cdot \mathbf n_F\{\eta\}
             -\mathbf T_F\cdot \mathbf n_F.
\]
Then $\|E_F\|_{L^\infty(F)}\le C(h\varepsilon_h+\varepsilon_h^2)
\le Ch\varepsilon_h$, and, because $\jp{W}=-\jp{\eta}$,
\begin{equation}\label{eq:residual-regular-flux}
 \FF(\phi)\cdot \mathbf n_F-\widehat F_{\mathbf n_F,F}(W^-,W^+)
       =a_0\{\eta\}+E_F+\tfrac12\nu_F\jp{\eta}.
\end{equation}
The trace $\jp{v}$ belongs to the tangential polynomial space of degree
$k$. Therefore $\int_F\{\eta\}\jp{v}=\int_F P_F\{\eta\}\jp{v}$.
For even $k$, Lemma~\ref{lem:parity} says
$\|P_F\{\eta\}\|_{L^\infty(F)}\le Ch^{k+2}=Ch\varepsilon_h$.
Define in this case
\[
 C_F=-a_0P_F\{\eta\}-E_F-\tfrac12\nu_F\jp{\eta}.
\]
The negative of the flux-difference integral in
\eqref{eq:residual-regular-flux} equals $\int_F C_F\jp{v}$.
It is this projected tested coefficient, rather than the unprojected
pointwise flux difference, that satisfies
$\|C_F\|_{L^\infty(F)}\le Ch\varepsilon_h$.

For odd $k$, the smooth extension has zero jumps of all normal
 derivatives used here. Substituting $W=\phi-\eta$ in the last term of
\eqref{eq:residual-exact-expanded} gives exactly
\[
 -S_h(W,v)=c_kh^k\sum_{F\in\mathcal F_h^r}\int_F
 a_F(W)\bigl(\jp{\partial_{\mathbf n_F}^k\eta}\jp{v}-\jp{\eta}\jp{\partial_{\mathbf n_F}^kv}\bigr).
\]
The coefficient of $\jp{v}$, after combining the first summand with
\eqref{eq:residual-regular-flux}, is
\begin{align}
 C_F={}&-a_0\bigl(P_F\{\eta\}-c_kh^k\jp{\partial_{\mathbf n_F}^k\eta}\bigr)
       -E_F-\tfrac12\nu_F\jp{\eta}\notag\\
      &+c_kh^k\bigl(a_F(W)-a_0\bigr)\jp{\partial_{\mathbf n_F}^k\eta}.
       \label{eq:residual-corrected-coefficient}
\end{align}
The derivative jump in this formula is tangentially polynomial:
The exact derivative cancels between its two traces. The four terms
are bounded as follows:
\begin{align*}
 \|P_F\{\eta\}-c_kh^k\jp{\partial_{\mathbf n_F}^k\eta}\|_{L^\infty(F)}
       &\le Ch^{k+2},\\
 \|E_F\|_{L^\infty(F)}+\nu_F\|\jp{\eta}\|_{L^\infty(F)}
       &\le Ch\varepsilon_h,\\
 \|a_F(W)-a_0\|_{L^\infty(F)}&\le C(h+\varepsilon_h),\\
 h^k\|\jp{\partial_{\mathbf n_F}^k\eta}\|_{L^\infty(F)}&\le C\varepsilon_h.
\end{align*}
Thus $\|C_F\|_{L^\infty(F)}\le Ch\varepsilon_h$ for odd $k$ as well.
For either parity, weighted Cauchy--Schwarz and the skeleton measure
bound yield
\begin{align}
 \left|\sum_{F\in\mathcal F_h^r}\int_F C_F\jp{v}\right|
 &\le\left(\sum_{F\in\mathcal F_h^r}
           \nu_F^{-1}\|C_F\|_F^2\right)^{1/2}
       \left(\sum_{F\in\mathcal F_h^r}\nu_F\|\jp{v}\|_F^2\right)^{1/2}
       \notag\\
 &\le Ch\varepsilon_h\left(\frac{Ch^{-1}}{\nu_0h}\right)^{1/2}
       \Dh(v)^{1/2}
 \le C\varepsilon_h\Dh(v)^{1/2}.
 \label{eq:residual-weighted-expanded}
\end{align}
For even degrees, $C_F$ denotes the corresponding coefficient without
the correction terms.

The second odd-degree summand contains a derivative of the test
function. Scaling the fixed space on a regular square gives
\[
 \left(\sum_{F\in\mathcal F_h^r}
       \|\jp{\partial_{\mathbf n_F}^kv}\|_F^2\right)^{1/2}
       \le Ch^{-k-1/2}\|v\|.
\]
The improved odd-degree bound $\|\jp{\eta}\|_{L^\infty(F)}\le Ch^{k+2}$
is essential at this point. It implies
\begin{align}
 \left|c_kh^k\sum_{F\in\mathcal F_h^r}\int_F
            a_F(W)\jp{\eta}\jp{\partial_{\mathbf n_F}^kv}\right|
 &\le Ch^k h^{k+2}
       \left(\sum_{F\in\mathcal F_h^r}|F|\right)^{1/2}
       \left(\sum_{F\in\mathcal F_h^r}
                       \|\jp{\partial_{\mathbf n_F}^kv}\|_F^2\right)^{1/2}
       \notag\\
 &\le Ch^k h^{k+2}h^{-1/2}h^{-k-1/2}\|v\|
   =C\varepsilon_h\|v\|.\label{eq:residual-derivative-expanded}
\end{align}

\paragraph{\bf Exceptional faces and the physical boundary.}
For $F\in \mathcal F_h^e\cup \mathcal F_h^b$, $\nu_F=\nu_0$, let $W^+=v^+=0$ on $F\in \mathcal F_h^b$ and define
\[\delta_s=(1-s)\eta^-+ s\eta^+,\quad q_F=\FF(\phi)\cdot \mathbf {n}_F-\widehat F_{\mathbf n_F,F}(W^-,W^+).\]
Thus $(1-s)W^-+sW^+=\phi-\delta_s$, by the integration-by-part, we have
\begin{align*}
  q_F&=\int_0^1\int_0^1\FF'(\phi-t\delta_s)\cdot \mathbf{n}_F \delta_s \,dt\,ds+ \frac{\nu_0}{2}\jp{\eta}\\
  &\leq C(|\eta^-|+|\eta^+|)\leq C \varepsilon_h.
\end{align*}
By the Cauchy-Schwarz inequality and \eqref{eq:geometry-strip}, we have
\begin{align}
    -\sum_{F\in \mathcal{F}_h^e\cup \mathcal F_h^b}\int_{F}q_F\jp{v}&\leq \left(\sum_{F\in \mathcal{F}_h^e\cup \mathcal F_h^b}\nu_0^{-1} \|q_F\|_{L^2(F)}^2\right)^{1/2}\times \left(\sum_{F\in \mathcal{F}_h^e\cup \mathcal F_h^b}\nu_0\|\jp{v}\|_{L^2(F)}^2\right)^{1/2} \notag\\
    &\leq C\varepsilon_h \left(\sum_{F\in \mathcal{F}_h^e\cup \mathcal F_h^b}|F|\right)^{1/2}\Dh(v)^{1/2}\leq C \varepsilon_h \Dh(v)^{1/2}\label{est:FeFb}
\end{align}

Combining this with
\eqref{eq:residual-volume-expanded},
\eqref{eq:residual-weighted-expanded},
\eqref{eq:residual-derivative-expanded} proves the weak estimate.
Finally, the face trace inequality and $\nu_F\le C$ imply
$\Dh(v)\le Ch^{-1}\|v\|^2$. For $r\ne0$, choose
$v=r/\|r\|$ in the weak estimate. This gives
$\|r\|\le C\varepsilon_h(1+h^{-1/2})\le C\varepsilon_h h^{-1/2}$;
The assertion is immediate when $r=0$.
\end{proof}

\subsection{Energy identities and semidiscrete convergence}
\begin{lemma}[Exact and relative energy]\label{lem:relative}
For every $w\in\Vh^k$,
\begin{equation}\label{eq:exact-energy}
 \ip{\Lh(w)}{w}=-\tfrac12\Dh(w).
\end{equation}
For $W=\Ph\phi$ as in Lemma~\ref{lem:weak-residual} and every $e\in\Vh^k$,
\begin{equation}\label{eq:relative-energy}
 \left|\ip{\Lh(W+e)-\Lh(W)}{e}+\tfrac12\Dh(e)\right|
 \leq C\norm{e}^2.
\end{equation}
The constant is independent of the size of $e$ for the extended flux.
\end{lemma}

\begin{proof}
\leavevmode\par
\paragraph{\bf Exact energy.}
Define the vector-valued primitive $\mathbf \Psi(s)=\int_0^s\FF(z)\,dz$.
The fundamental theorem of calculus, also at $a=b$ by continuity,
gives
\[
 (a-b)\FF^{\rm ec}(a,b)=\mathbf \Psi(a)-\mathbf \Psi(b),\qquad
 \FF(w)\cdot\nabla w=\nabla\cdot\mathbf \Psi(w).
\]
Summing integration by parts over the cells, the interior boundary
term is $\sum_F\int_F(\mathbf\Psi(w^-)-\mathbf\Psi(w^+))\cdot \mathbf n_F$.
It cancels the entropy-conservative numerical flux tested against
$\jp{w}$. The exterior boundary term cancels as well, since
$w\F^{\rm ec}(w,0)=\mathbf \Psi(w)-\mathbf \Psi(0)=\Psi(w)$.
Each viscosity term contributes $-\nu_F\|\jp{w}\|_F^2/2$,
and the boundary contributes $-\nu_0\|w\|_\Gamma^2/2$.
The skew integrand is zero when its two arguments agree. These
identities prove \eqref{eq:exact-energy}.

\paragraph{\bf A relative identity without divided differences.}
Define
\begin{align*}
 \mathbf H(W,e)&=\mathbf \Psi(W+e)-\mathbf \Psi(W)-\mathbf \F(W)e,\\
 \mathbf Q(W,e)&=\mathbf \F(W+e)-\F(W)-\mathbf \F'(W)e
       =\int_0^1(1-t)\mathbf \F''(W+te)e^2\,dt.
\end{align*}
Direct differentiation gives
$\partial_W \mathbf H=\mathbf Q(W,e)$ and
$\partial_e \mathbf H=\mathbf \F(W+e)-\mathbf \F(W)$. Therefore, the exact cell identity is
\begin{equation}\label{eq:relative-cell-expanded}
 (\mathbf \F(W+e)-\mathbf \F(W))\cdot\nabla e
  =\nabla\cdot \mathbf H(W,e)-\mathbf Q(W,e)\cdot\nabla W.
\end{equation}
On a face let $a=W^-$, $b=W^+$, $p=e^-$, $q=e^+$, and set
$W_s=(1-s)a+sb$, $e_s=(1-s)p+sq$. The chain rule gives
\[
 \frac{d}{ds}\mathbf H(W_s,e_s)
 =-(a-b)\mathbf Q(W_s,e_s)
   -(p-q)\bigl(\mathbf \F(W_s+e_s)-\mathbf \F(W_s)\bigr).
\]
Integrating from zero to one and rearranging proves exactly
\begin{align}
 &\mathbf H(a,p)-\mathbf H(b,q)
 -(p-q)\bigl(\mathbf \F^{\rm ec}(a+p,b+q)-\mathbf \F^{\rm ec}(a,b)\bigr)
 \notag\\
 &\hspace{30mm}=(a-b)\int_0^1\mathbf Q(W_s,e_s)\,ds.
 \label{eq:relative-face}
\end{align}
No division by $a-b$ was used. This identity therefore includes
$a=b$. At a boundary face, take $b=q=0$; then $\mathbf H(0,0)=0$ and
\eqref{eq:relative-face} is the corresponding boundary identity.

Let $W^+=e^+=0$ on the boundary face. Combining
\eqref{eq:relative-cell-expanded}--\eqref{eq:relative-face} gives
\begin{align}
 &(\mathcal L_h(W+e)-\mathcal L_h(W),e)
       +\frac12\Dh(e)\notag\\
 &=-\sum_{K\in \cM}\int_K \mathbf Q(W,e)\cdot\nabla W
   +\sum_{F\in\mathcal F_h}\int_F\jp{W}
                     \int_0^1\mathbf Q(W_s,e_s)\cdot\mathbf  n_F\,ds
   +\mathcal S(W,e),\label{eq:relative-exact-expanded}
\end{align}
where $\mathcal S(W,e)=S_h(W+e,e)-S_h(W,e)$.
The bound $|\mathbf Q(W,e)|\le \frac12\|\mathbf \F''\|_\infty |e|^2$ and
$\|\nabla W\|_\infty\le C$ control the volume term by $C\|e\|^2$.
For the faces, define $\sigma_F(e)=|e^-|+|e^+|$. Projection
approximation and continuity of the exact trace give
$|\jp{W}|\le C\varepsilon_h$ on $\mathcal F_h$, including $\Gamma$.
Since $|e_s|\le\sigma_F(e)$, the face term is bounded by
\begin{equation}\label{eq:relative-face-bound-expanded}
 C\varepsilon_h\sum_{F\in\mathcal F_h}\|\sigma_F(e)\|_F^2
 \le C\varepsilon_h h^{-1}\|e\|^2\le C\|e\|^2.
\end{equation}

\paragraph{\bf The odd-degree term.}
For odd $k$ define the pointwise skew expression
\[
 B_F(z,v)=c_kh^k\bigl(\jp{\partial_{\mathbf n_F}^kz}\jp{v}-\jp{z}\jp{\partial_{\mathbf n_F}^kv}\bigr).
\]
Bilinearity and $B_F(e,e)=0$ imply
\[
 \mathcal S(W,e)=\sum_{F\in\mathcal F_h^r}\int_F
        (a_F(W+e)-a_F(W))B_F(W,e).
\]
The mean-value formula gives
$|a_F(W+e)-a_F(W)|\le C\sigma_F(e)$.
The parity lemma, applied to the smooth reference, gives
\[
 h^k|\jp{\partial_{\mathbf n_F}^kW}|\le C\varepsilon_h,\qquad
 |\jp{W}|\le Ch^{k+2}\quad\hbox{on regular faces}.
\]
Thus the first part of $\mathcal S$ is at most
$C\varepsilon_h\sum_F\|\sigma_F(e)\|_F^2\le Ch^k\|e\|^2$.
For the second part, Cauchy--Schwarz and derivative traces give
\begin{align*}
 Ch^k h^{k+2}\sum_{F\in \mathcal{F}_h^r}\int_F\sigma_F(e)|\jp{\partial_{\mathbf n_F}^ke}|
 &\le Ch^{2k+2}(Ch^{-1/2}\|e\|)(Ch^{-k-1/2}\|e\|)\\
 &\le Ch^{k+1}\|e\|^2.
\end{align*}
For even $k$, $\mathcal S=0$. Substitution in
\eqref{eq:relative-exact-expanded} proves the asserted two-sided
bound.
\end{proof}

\paragraph{\bf The proof of Theorem 3.1}
The semidiscrete equation is $U_t=\mathcal L_h(U)$ and $U(0)=P_hu_0$.
Testing with $U$ yields
$\frac{d}{dt}\|U\|^2=-\Dh(U)$.
Integration from $0$ to $t$ proves the exact norm identity in \eqref{eq:semidiscrete-energy}.

Put $W=P_hu$, $e=U-W$, and
$r=P_hu_t-\mathcal L_h(W)=r(u)$. The projector is independent of
time, so
\[
 e_t=\mathcal L_h(W+e)-\mathcal L_h(W)-r,\qquad e(0)=0.
\]
Let $E(t)=\|e(t)\|^2$. The preceding two lemmas imply
\[
 \frac12E_t+\frac12\Dh(e)
 \le C_0E+C_1\varepsilon_h E^{1/2}
                +C_1\varepsilon_h\Dh(e)^{1/2}.
\]
Use $ab\le(a^2+b^2)/2$ on the middle term and
$ab\le a^2/4+b^2$ on the last term, taking
$a=\Dh(e)^{1/2}$ there. For fixed constants $C_2,C_3$,
\begin{equation}\label{eq:semidiscrete-differential-expanded}
 \frac12E_t+\frac14\Dh(e)
       \le C_2E+C_3\varepsilon_h^2.
\end{equation}
Dropping the nonnegative dissipation and multiplying by
$\exp(-2C_2t)$ gives
\[
 \frac{d}{dt}\bigl(e^{-2C_2t}E(t)\bigr)
       \le2C_3\varepsilon_h^2e^{-2C_2t}.
\]
Since $E(0)=0$, integration implies
$E(t)\le2C_3T e^{2C_2T}\varepsilon_h^2$ for $0\le t\le T$.
Integrating \eqref{eq:semidiscrete-differential-expanded}, using
this bound and $E(T)\ge0$, also gives
$\int_0^T\Dh(e)\,dt\le C_T\varepsilon_h^2$.
Finally,
$\|U-u\|\le\|e\|+\|P_hu-u\|\le C_T\varepsilon_h$.

For completeness, these estimates apply throughout $[0,T]$.
For fixed $h$ the mass matrix is positive definite, and the
extended-flux vector field is smooth in the finite-dimensional
coefficient vector. Local existence and uniqueness therefore hold.
If the maximal existence time were finite, the exact energy identity
would keep $U$ in a bounded coefficient set, which is compact in
finite dimension. The smooth vector field is bounded and Lipschitz
on a neighborhood of that set, permitting continuation beyond the
maximal time, a contradiction. The extension consequently has a
global solution. Moreover
\[
 \|U-u\|_{L^\infty(\Omega)}
 \le Ch^{-1}\|e\|+\|P_hu-u\|_{L^\infty(\Omega)}
 \le C_T(h^k+h^{k+1})\longrightarrow0.
\]
Choose $h$ so this bound is smaller than the fixed margin of the
flux agreement interval. The extended and original
operators then coincide along the whole solution, proving the
claim for the original smooth flux.
\begin{remark}
To suppress spurious oscillations near discontinuities, we augment
the DG discretization~\eqref{eq:operator} with a damping term.
Specifically, the OFDG operator
$\Lh^{\rm OF}:\Vh^k\to\Vh^k$ is defined by
\begin{equation}\label{eq:operator_OFDG}
 \ip{\Lh^{\rm OF}(w)}{v}
 =
 \ip{\Lh(w)}{v}-\Sigma_h(w,v),
 \qquad \forall v\in\Vh^k,
\end{equation}
where
\begin{equation}\label{eq:damping_term}
 \Sigma_h(w,v)
 =
 \sum_{K\in\cM}\sum_{\ell=0}^{k}
 \frac{\sigma_K^\ell(w)}{h}
 \int_K \bigl(w-P_h^{\ell-1}w\bigr)v\,dx.
\end{equation}
The damping coefficient $\sigma_K^\ell(w)$ is given by
\begin{equation}\label{damping_coeff}
 \sigma_K^\ell(w)
 =
 \frac{2(2\ell+1)}{2k-1}\frac{h^\ell}{\ell!}
 \sum_{|\bm{\alpha}|=\ell}
 \left(
   \frac{1}{N_e}
   \sum_{\mathbf{v}\in\mathcal V(K)}
   \left(
     \left.\jp{\partial^{\bm{\alpha}}w}\right|_{\mathbf{v}}
   \right)^2
 \right)^{1/2}.
\end{equation}
Here, $\bm{\alpha}$ is a multi-index, $\mathcal V(K)$ denotes
the set of vertices of $K$, and $N_e$ is the number of edges
of $K$. The quantity $\left.\jp{w}\right|_{\mathbf{v}}$ denotes
the jump of $w$ at the vertex $\mathbf{v}$ between $K$ and its
neighboring elements. For $\ell\ge0$, $P_h^\ell$ denotes the
standard cellwise $L^2$ projection onto $\Vh^\ell$, with the
convention $P_h^{-1}=P_h^0$.

By arguments analogous to those in~\cite{LuLiuShu2021SINUM},
the semidiscrete OFDG scheme $U_t=\Lh^{\rm OF}(U)$ retains
conservation and $L^2$ stability and achieves an optimal
$O(h^{k+1})$ error estimate for smooth solutions.
The numerical experiments include accuracy tests for smooth
solutions and tests involving discontinuities to assess
the convergence rates and the suppression of spurious oscillations.
\end{remark}

\subsection{Linear advection with inflow data}\label{sec:linear-inflow}
For constant vector $\boldsymbol\beta$, the interior method also admits the
usual inflow boundary condition. Put $b=\boldsymbol\beta\cdot \mathbf n$,
$b^+=\max(b,0)$, $b^-=\min(b,0)$, and prescribe $g$ on
$\Gamma_-=\{b<0\}$. Replace the boundary flux by
$b^+w+b^-g$. Write the resulting affine operator as
$A_hw+\ell_h(t)$ and set
\[
 \Dh^{\rm up}(v)=\sum_{F\in\Fh^r\cup\Fh^e}\nu_F\norm{\jp{v}}_F^2
                 +\sum_{F\in \Fh^b}\int_F|b|v^2.
\]
\begin{theorem}\label{thm:linear-inflow}
On the mesh of Section~\ref{sec:geometry}, let $u$ solve
$u_t+\boldsymbol\beta\cdot\nabla u=0$ with inflow trace $g$ and
a bounded $C^{k+2}$ space--time extension. The semidiscrete method initialized by
$\Ph u_0$ satisfies $\sup_{t\leq T}\norm{U-u}\leq C_T h^{k+1}$
for every $k\geq1$. Its exact energy identity is
\begin{equation}\label{eq:inflow-energy}
 \frac12\frac{d}{dt}\norm{U}^2+\frac12\Dh^{\rm up}(U)
 =\int_{\Gamma_-}|b|gU.
\end{equation}
In particular,
\[
 \norm{U(t)}^2\le\norm{\Ph u_0}^2
                +\int_0^t\int_{\Gamma_-}|b|g^2\,ds\,dt'.
\]
\end{theorem}
\begin{proof}
Interior cancellation and boundary integration by parts give
$\ip{A_hv}{v}=-\Dh^{\rm up}(v)/2$ and
$\ip{\ell_h(t)}{v}=\int_{\Gamma_-}|b|gv$, proving
\eqref{eq:inflow-energy}. Completing the inflow square gives
\begin{align*}
 \frac12\frac{d}{dt}\norm{U}^2
 &+\frac12\sum_{F\in\Fh^r\cup\Fh^e}\nu_F\norm{[U]}_F^2
 +\frac12\int_{\Gamma_+}bU^2\\
 &+\frac12\int_{\Gamma_-}|b|(U-g)^2
 =\frac12\int_{\Gamma_-}|b|g^2,
\end{align*}
where $\Gamma_+=\{b>0\}$. Integration proves the stated energy bound.
With $W=\Ph u$, the exact boundary-flux
difference is
$bu-(b^+W+b^-g)=b^+(u-W)$, since $u=g$ on inflow.
Consequently, its tested residual is at most
\[
 \left(\int_\Gamma b^+(u-W)^2\right)^{1/2}
 \left(\int_\Gamma b^+v^2\right)^{1/2}
 \leq C\epsh\Dh^{\rm up}(v)^{1/2}.
\]
The interior proof of Lemma~\ref{lem:weak-residual} is unchanged.
Thus $e=U-\Ph u$ satisfies
$\tfrac12 \frac{d}{dt}\norm{e}^2+\tfrac12\Dh^{\rm up}(e)=-(r,e)$;
Young's inequality and Gronwall's inequality prove the claim.
\end{proof}

\section{Numerical experiments}\label{sec:numerics}

We test the method for a nonlinear equation on a disk and for linear
advection on a smooth nonconvex domain. Both computations use a uniform
Cartesian background mesh, the merging parameter $\delta_0=0.2$, and
$\nu_0=0.1$. The approximation space is $\mathbb Q_k$ on each physical
cell, with $k=1,2,3$. The correction $S_h$ is included for $k=1,3$;
it vanishes for $k=2$. No limiter or filter is applied.
The initial value is the physical $L^2$ projection.
Tensor Legendre bases are scaled to the Cartesian bounding box of
each physical cell, and the full physical mass matrix is used.
This rescaling changes the basis without changing the polynomial space. The time discretization is the explicit third-order Runge-Kutta method.  

The boundary is represented by its analytic parametrization.
Polynomial integrals on ordinary rectangles and straight faces are
evaluated with $\lceil(3k+1)/2\rceil$ Gaussian points per coordinate
for Burgers' equation and $k+2$ points for linear advection.
These rules integrate the polynomial operator products exactly.
Integrals on curved
cells are evaluated by high-order quadrature on smooth parameter
intervals; there is no polygonal approximation of the boundary.
These computations therefore use a numerical approximation to the
exact integration assumed in the analysis. Quadrature refinement is
reported below. Errors are evaluated against the analytic solution
with a separate higher-order rule.
The observed order between consecutive meshes is
\begin{equation}\label{eq:numerical-rate}
 p_j=\frac{\log(E_{j-1}/E_j)}{\log(h_{j-1}/h_j)},
 \qquad E_j=\|u(T)-u_h(T)\|_{L^2(\Omega)}.
\end{equation}
The reported computations were partially carried out using the
high-performance computing facilities of the State Key Laboratory of Scientific and Engineering Computing, Chinese Academy of Sciences.
\subsection{Burgers equation on the unit disk}\label{subsec:burgers-disk}
Consider
\begin{equation}\label{eq:disk-test}
 \begin{split}
 u_t+\partial_x(u^2/2)+\partial_y(u^2/2)=0,
 \qquad &\Omega=\{(x,y):x^2+y^2<1\},\\
 &u(0,x,y)=1-x^2-y^2.
 \end{split}
\end{equation}
We use the homogeneous boundary trace and compute to $T=0.2$.
The exact solution is
\begin{equation}\label{eq:disk-exact}
 u(t,x,y)=\frac{2d}{q+\sqrt{q^2+8t^2d}},
 \qquad d=1-x^2-y^2,\quad q=1-2t(x+y).
\end{equation}
Indeed, the characteristics satisfy
$(x,y)=(\xi,\eta)+t(1-\xi^2-\eta^2)(1,1)$, which gives
$2t^2u^2+qu-d=0$.
In the coordinates $z=(\xi+\eta)/\sqrt2$ and
$w=(\xi-\eta)/\sqrt2$, the characteristic map on a disk chord is
$z\mapsto z+\sqrt2t(1-w^2-z^2)$. Its derivative is at least
$1-2\sqrt2T>0$, and it fixes both endpoints.
Thus the solution is smooth on the closed disk throughout the test.
Although its boundary trace vanishes, its normal derivative is
$-2/[1-2t(x+y)]$ on the circle, so the approximation error is not
confined away from the cut cells.

The background box is $[-1.2,1.2]^2$, with $h=2.4/N$ and
$N=16,32,64,128$. The corresponding numbers of physical cells are
$148,560,2268,8976$. Figure ~\ref{fig:disk-mesh} shows the merged mesh. On the finest mesh, the smallest unmerged
intersection has area $0.00072507h^2$, whereas the smallest final
physical cell has area $0.033601h^2$.
The normal numerical flux specializes to
\begin{equation}\label{eq:burgers-test-flux}
 \widehat F_{\mathbf n_F,F}(a,b)
 =(n_x+n_y)\frac{a^2+ab+b^2}{6}+\frac{\nu_F}{2}(a-b).
\end{equation}
The coefficient in $S_h$ is $(n_x+n_y)\{u_h\}$.
The time step is $\tau=T/m$, where
$m=\lceil T/(0.04h^{\max(1,(k+1)/3)})\rceil$. This choice makes the formal third-order time error smaller
than the spatial error.
\begin{figure}[tb]
\centering
\includegraphics[width=0.8\textwidth]{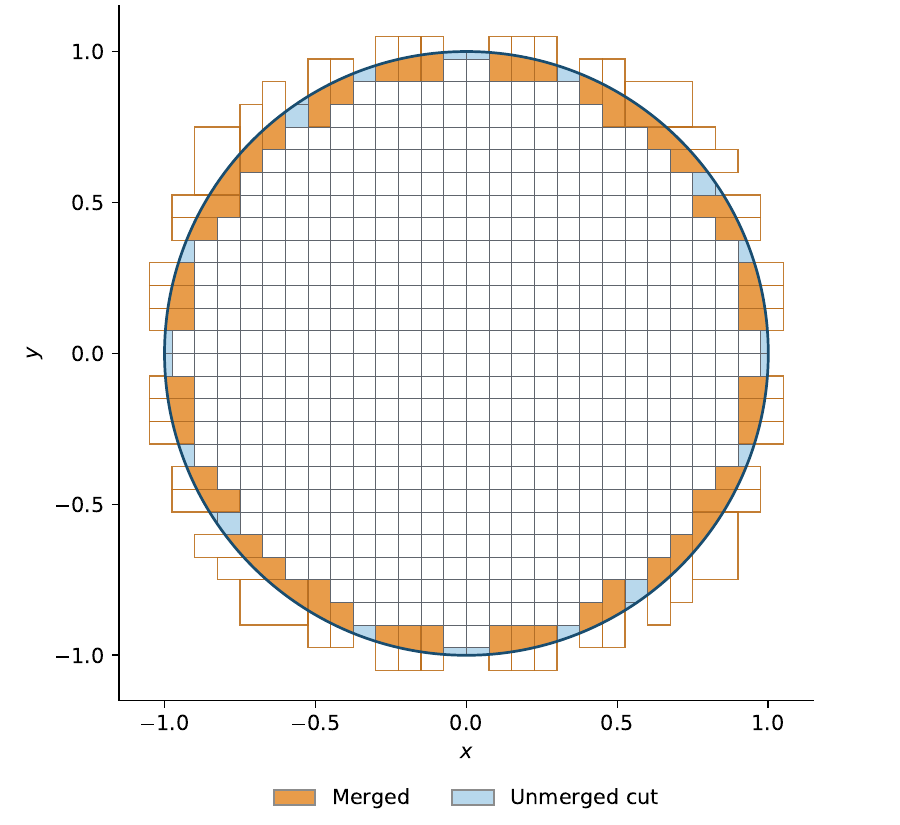}
\caption{ Unit disk domain with $N=32$.}
\label{fig:disk-mesh}
\end{figure}
\begin{table}[tb]
\centering\small\setlength{\tabcolsep}{4pt}
\caption{Burgers equation on the disk: $L^2$ errors at $T=0.2$.}
\label{tab:disk-errors}
\begin{tabular}{r rr rr rr}
\toprule
&\multicolumn{2}{c}{$k=1$}&\multicolumn{2}{c}{$k=2$}&\multicolumn{2}{c}{$k=3$}\\
$N$&Error&Order&Error&Order&Error&Order\\
\midrule
16  &$1.1215\times10^{-2}$&--    &$6.5156\times10^{-4}$&--    &$7.5471\times10^{-5}$&--\\
32  &$2.4281\times10^{-3}$&2.208 &$7.5813\times10^{-5}$&3.103 &$4.9260\times10^{-6}$&3.937\\
64  &$5.0067\times10^{-4}$&2.278 &$7.8163\times10^{-6}$&3.278 &$2.6164\times10^{-7}$&4.235\\
128 &$1.2016\times10^{-4}$&2.059 &$1.0338\times10^{-6}$&2.918 &$1.8487\times10^{-8}$&3.823\\
\bottomrule
\end{tabular}
\end{table}

Table~\ref{tab:disk-errors} shows second-, third-, and fourth-order
spatial convergence. The endpoint orders from $N=16$ to $128$ are
$2.181$, $3.100$, and $3.998$.
The individual rates vary as the cut positions and merging patterns
change. 
Increasing the curved quadrature from $24$ to $48$ angular points
per smooth piece changes the $k=3$, $N=16$ solution by
$3.6517\times10^{-14}$.

\subsection{Burgers equation on a disk with inflow boundary conditions}
\label{subsec:burgers-disk-inflow}

We again consider the two-dimensional Burgers equation on a disk,
now with inflow boundary data prescribed by the corresponding
whole-space entropy solution:
\begin{equation}\label{2D Burgers equation}
\begin{cases}
u_t + \frac{1}{2}(u^2)_x + \frac{1}{2}(u^2)_y = 0,
& (x,y)\in\Omega,\quad t>0, \\[6pt]
u(x,y,0) = \frac34 + \frac 12\sin(\pi(x+y)),
& (x,y)\in\bar{\Omega}, \\[6pt]
u(x,y,t) = g(x,y,t),
& (x,y)\in\Gamma,\quad t>0,
\end{cases}
\end{equation}
where
\begin{equation}\label{2D_Wave_Disk}
\Omega = \{(x,y)\in\mathbb R^2 : x^2+y^2<\frac 12\},
\qquad
\Gamma = \{(x,y)\in\partial\Omega : x+y\le0\}.
\end{equation}
The inflow boundary data are given by $g(x,y,t)=u_\star(x,y,t)$,
where $u_\star$ denotes the entropy solution of the corresponding
Cauchy problem on $\mathbb R^2$ with the same initial profile.

The background mesh is constructed on the square
$[-1.2/\sqrt{2},1.2/\sqrt{2}]^2$, with mesh size
$h=2.4/(\sqrt{2}N)$ for $N=16,32,64,128$.
Figure~\ref{fig:burgers_accuracy} shows that both the DG and OFDG
methods achieve the expected second-, third-, and fourth-order
convergence rates for $k=1,2,3$, respectively, at $T=0.1$.
In the presence of discontinuities, the OFDG method effectively
captures shocks without noticeable spurious oscillations,
as shown in Figure~\ref{fig:ofdg_shock}.

\begin{figure}[htbp]
\centering
\includegraphics[width=\linewidth]{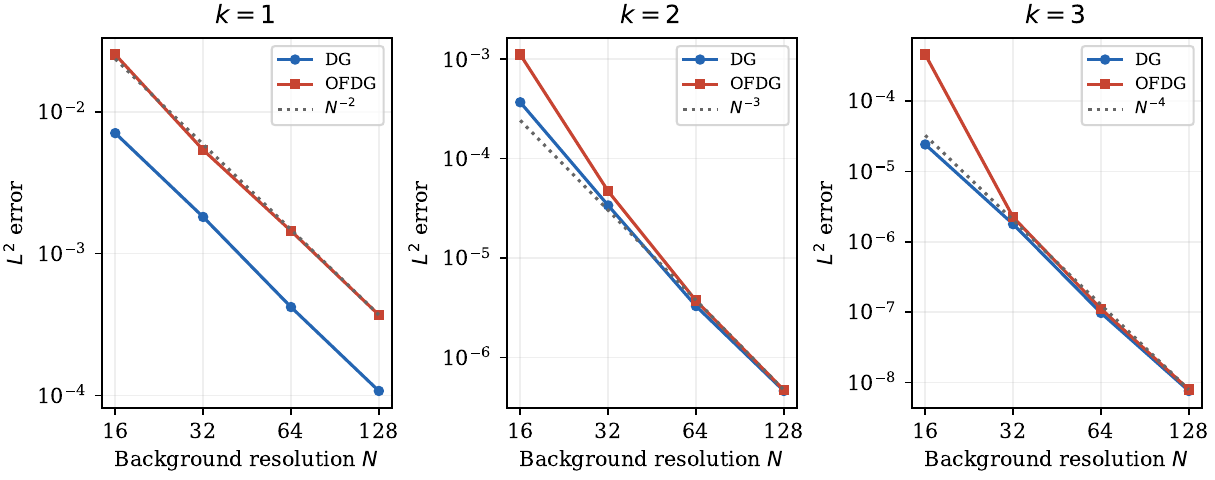}
\caption{Smooth DG/OFDG convergence at $T=0.1$.}\label{fig:burgers_accuracy}
\end{figure}

\begin{figure}[htbp]
\centering
\includegraphics[width=0.3\linewidth]{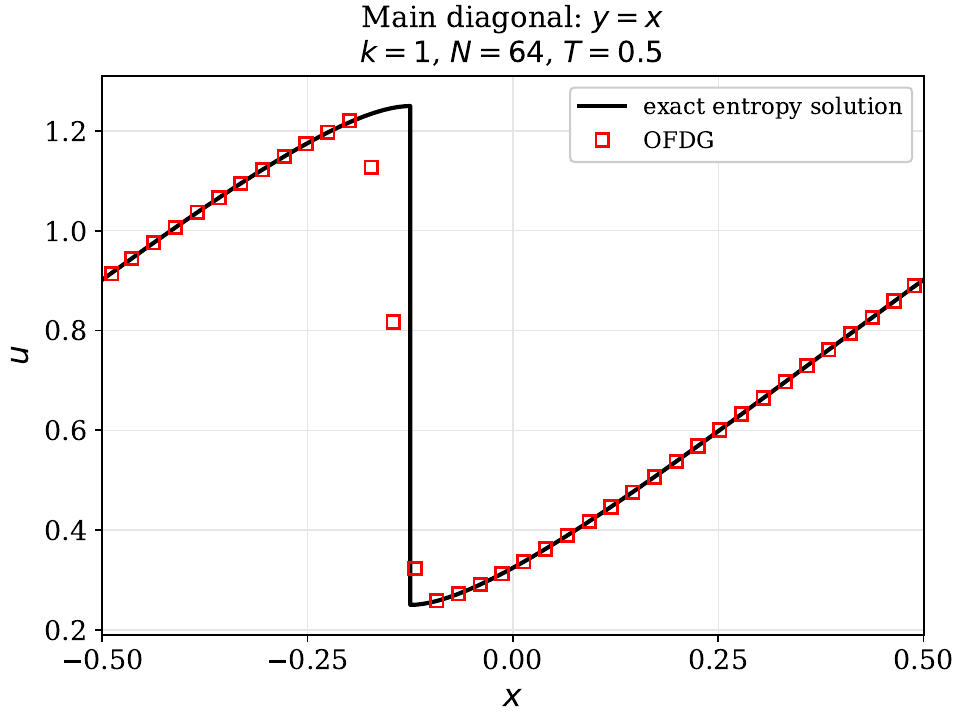
}
\includegraphics[width=0.3\linewidth]{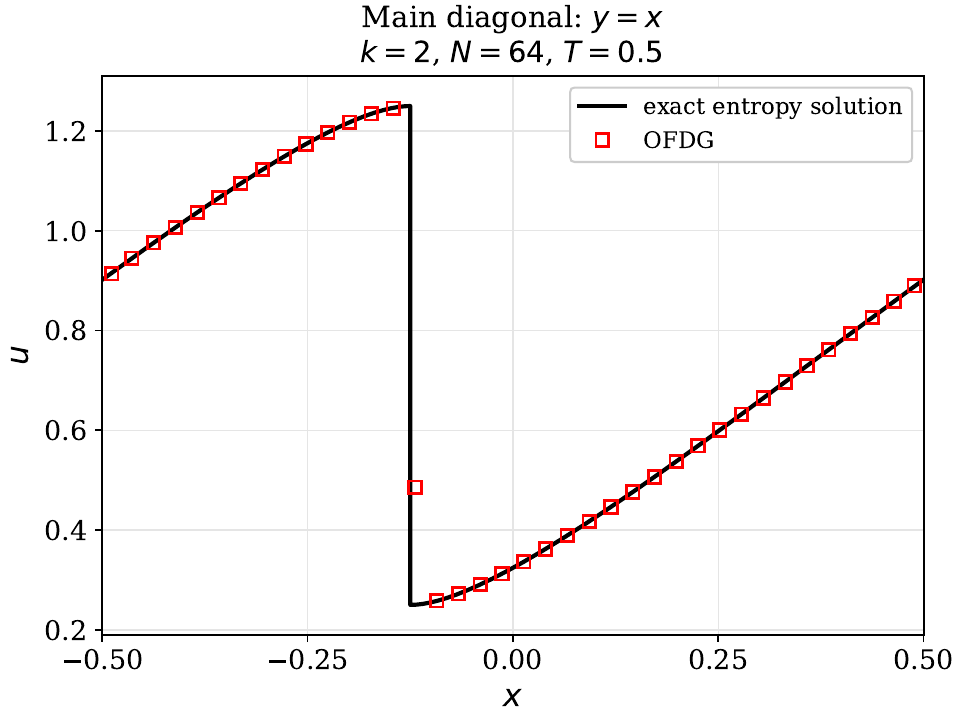
}
\includegraphics[width=0.3\linewidth]{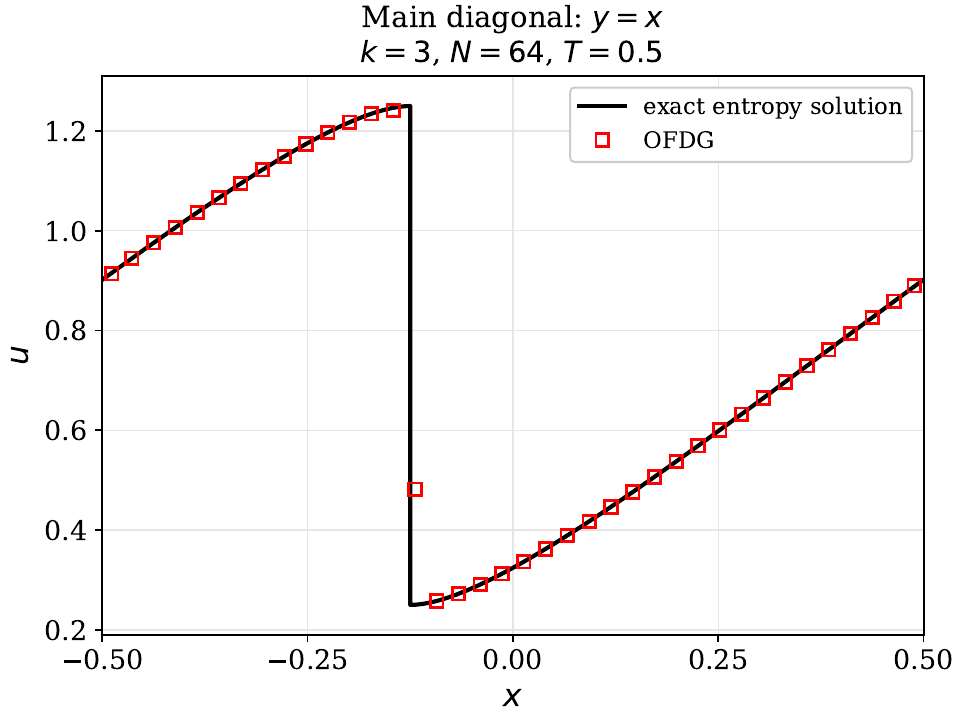
}
\caption{Shock diagonal cuts for $k=1,2,3$(from left to right), $N=64$, $T=0.5$.}\
\label{fig:ofdg_shock}
\end{figure}

\subsection{Linear advection on a smooth five-pointed star}
\label{subsec:linear-star}
We next consider $u_t+u_x+u_y=0$ on the domain:
\begin{equation}\label{eq:star-domain}
 \begin{split}
 \Omega&=\{(r\cos\theta,r\sin\theta):0\le r<R(\theta),\quad
                   0\le\theta<2\pi\},\\
 R(\theta)&=\frac29\bigl(3+4^{\sin(5\theta)}\bigr).
 \end{split}
\end{equation}
Figure~\ref{fig:star-mesh} shows the merged mesh and its concave boundary.
\begin{figure}[tb]
\centering
\includegraphics[width=\textwidth]{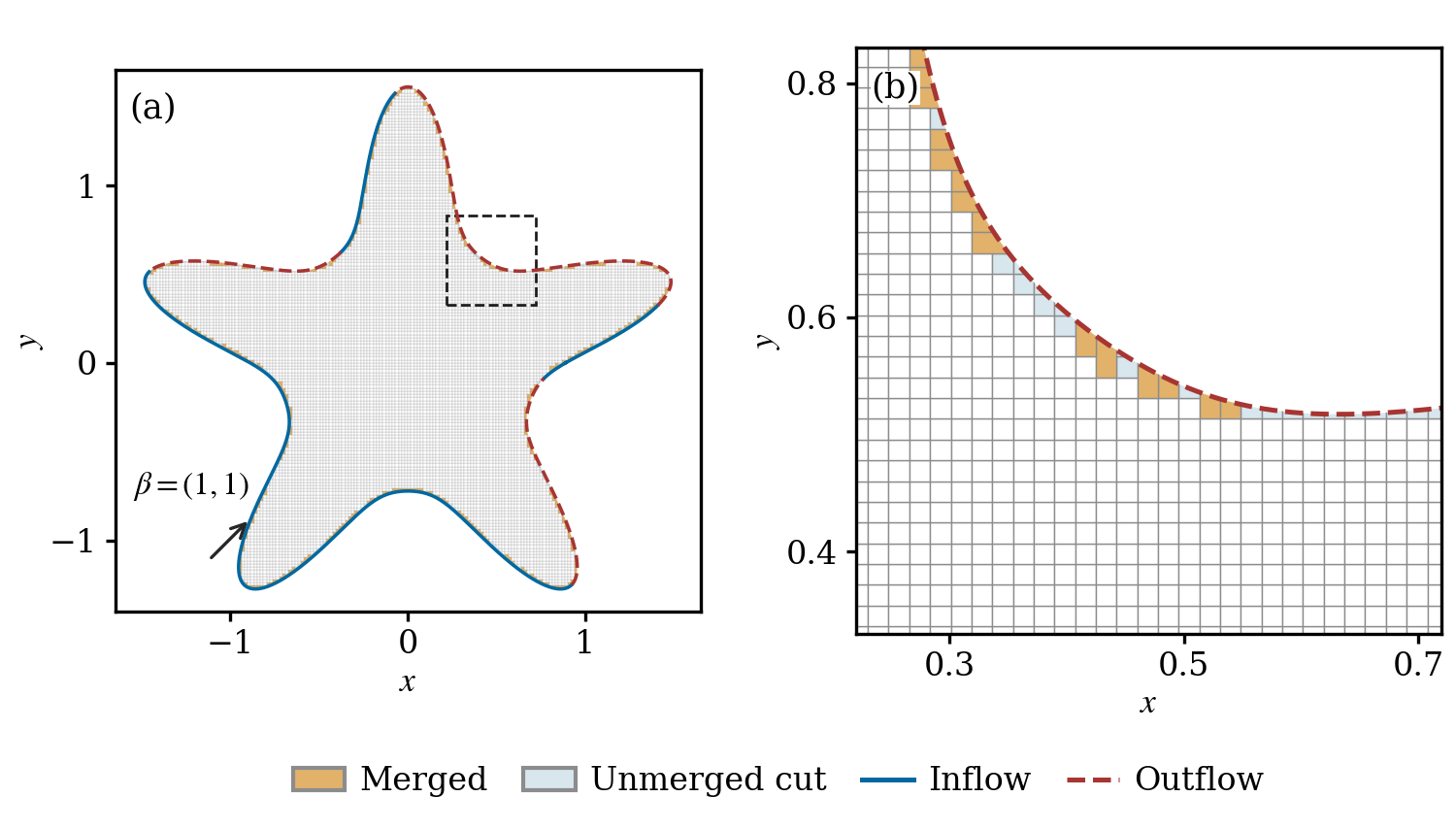}
\caption{Star domain with $N=226$: (a) merged mesh and (b) enlargement
of the boxed region. Solid blue and dashed
red curves indicate inflow and outflow, respectively.}
\label{fig:star-mesh}
\end{figure}
Set $\boldsymbol\beta=(1,1)$ and $T=0.1$. The exact solution is
\begin{equation}\label{eq:star-exact}
 u(t,x,y)=\frac54\sin\bigl(\pi(x+y-2t)\bigr)
                +\frac14\sin\bigl(\pi(y-x)\bigr).
\end{equation}
The second term is stationary, and the first travels in the
direction $\boldsymbol\beta$. We prescribe its trace $g$ on
$\Gamma_-:=\{\boldsymbol\beta\cdot\mathbf n<0\}$ and impose no data
on outflow. With $b=\boldsymbol\beta\cdot\mathbf n$, the physical boundary
flux is $b^+u_h+b^-g$, where $b^+=\max(b,0)$ and
$b^-=\min(b,0)$.
On interior faces, the flux is
$b_F\{u_h\}+\nu_F[u_h]/2$, and the coefficient in $S_h$ is
the constant $b_F=\boldsymbol\beta\cdot \mathbf{n}_F$.
This test has nonzero time-dependent inflow; its semidiscrete
energy balance is covered by Theorem~\ref{thm:linear-inflow}.

The box $[-2,2]^2$ is divided into $N^2$ squares, so $h=4/N$.
We use $N=226,320,452,640$, for which the admissibility conditions
of the merging construction hold. These grids contain
$11138,22296,44370,88890$ physical cells, respectively.
The successive refinement ratios are close to $\sqrt2$;
all rates in Table~\ref{tab:star-errors} use the actual ratios in
\eqref{eq:numerical-rate}.
At $N=320$, the smallest original intersection has area
$5.51\times10^{-8}h^2$ and lies in a merged physical cell of
area $1.53h^2$. The smallest final physical cell on this mesh
has area $0.02431h^2$.

\begin{table}[tb]
\centering\small\setlength{\tabcolsep}{4pt}
\caption{Linear advection on the star: $L^2$ errors at $T=0.1$.
Orders use the actual background mesh ratios.}
\label{tab:star-errors}
\begin{tabular}{r rr rr rr}
\toprule
&\multicolumn{2}{c}{$k=1$}&\multicolumn{2}{c}{$k=2$}&\multicolumn{2}{c}{$k=3$}\\
$N$&Error&Order&Error&Order&Error&Order\\
\midrule
226 &$5.9156\times10^{-4}$&--    &$2.4486\times10^{-6}$&--    &$1.7888\times10^{-8}$&--\\
320 &$2.9414\times10^{-4}$&2.009 &$9.1623\times10^{-7}$&2.827 &$4.5599\times10^{-9}$&3.930\\
452 &$1.4578\times10^{-4}$&2.032 &$3.2312\times10^{-7}$&3.018 &$1.0955\times10^{-9}$&4.129\\
640 &$7.2309\times10^{-5}$&2.016 &$1.0817\times10^{-7}$&3.146 &$2.6304\times10^{-10}$&4.102\\
\bottomrule
\end{tabular}
\end{table}

The step is $\tau=T/\lceil T/(c_k^{\mathrm{time}}h^{\max(1,(k+1)/3)})\rceil$,
with $c_1^{\mathrm{time}}=c_2^{\mathrm{time}}=0.08$ and
$c_3^{\mathrm{time}}=0.04$.
Inflow data are evaluated at the RK stage times
$t_n$, $t_n+\tau$, and $t_n+\tau/2$.
The observed endpoint orders are $2.019$, $2.997$, and $4.054$.
Doubling the angular quadrature from $24$ to $48$ points for
$k=3$, $N=226$ changes the solution by $2.5624\times10^{-13}$.
Boundary quadrature intervals are split at grid crossings and at
changes of sign of $\boldsymbol\beta\cdot \mathbf n$.

\section{Concluding Remarks}\label{sec:concluding}
In this paper, we develop an unfitted discontinuous Galerkin method
for scalar hyperbolic conservation laws on curved domains. The method
employs a uniform Cartesian background mesh and a cell-merging
algorithm to address the small-cut-cell problem. To achieve
optimal-order error estimates on meshes with curved boundaries,
the semidiscrete formulation combines an entropy-conservative flux
with face viscosity and, for odd polynomial degrees, incorporates
an additional skew derivative-jump correction on regular faces.
The error analysis relies on the physical-cell $L^2$ projection
and the cancellation of projection-error contributions on regular
faces. The analysis also covers linear transport problems with
an upwind treatment of inflow boundary conditions. Our ongoing
work focuses on extending this framework to three-dimensional
problems and hyperbolic systems.

\section*{Declaration of Generative AI Use}
During the preparation of this manuscript, the author used ChatGPT
(OpenAI) to assist with language editing and revision of the \LaTeX{}
source. The author reviewed the resulting content and takes full
responsibility for the final manuscript.

\FloatBarrier
\bibliographystyle{amsplain}
\providecommand{\bysame}{\leavevmode\hbox to3em{\hrulefill}\thinspace}
\providecommand{\MR}{\relax\ifhmode\unskip\space\fi MR }
\providecommand{\MRhref}[2]{%
  \href{http://www.ams.org/mathscinet-getitem?mr=#1}{#2}
}
\providecommand{\href}[2]{#2}

\end{document}